\documentclass{article} 

\usepackage[usenames,dvipsnames]{color}
\usepackage[T1]{fontenc} 
\usepackage[utf8]{inputenc} 
\RequirePackage[numbers]{natbib}

\title{Coupling Polynomial Stratonovich Integrals: \\the two-dimensional Brownian case\footnote{This work was supported by EPSRC Research Grant EP/K013939.  
This is a theoretical research paper and, as such, no new data were created during this study.}
} 

\author{%
  Sayan~Banerjee\footnote{University of North Carolina, Chapel Hill, USA.
    \texttt{{sayan@email.unc.edu}}}
  ~and 
  Wilfrid~Kendall\footnote{University of Warwick, UK. \texttt{{w.s.kendall@warwick.ac.uk}}}
  }

\date{\today}

\RequirePackage{amsmath}
 \RequirePackage{amsthm}
 \RequirePackage{amsfonts}
 \RequirePackage{framed}
 \RequirePackage[small]{caption}
 \RequirePackage{xspace}

 \usepackage{a4wide}
 
\theoremstyle{plain}
 \newtheorem{thm}{Theorem}
 \newtheorem{lem}[thm]{Lemma}

\theoremstyle{definition}
 
\theoremstyle{remark}
 \newtheorem{rem}[thm]{Remark}

 \newcommand{\Expect}[1]{\operatorname{\mathbb{E}}\left[#1\right]}
 
 \newcommand{\Integers}{\mathbb{Z}}

 \newcommand{\Prob}[1]{\operatorname{\mathbb{P}}\left[#1\right]}
 
 \newcommand{\Reals}{\mathbb{R}}

 \renewcommand{\d}{\,\operatorname{d}}
 
 \newcommand{\eps}{\varepsilon}

 \DeclareMathOperator{\sgn}{sgn}

\newcommand{\X}{\mathbf{X}}

\newcommand{\tw}{\widetilde{w}}

\newcommand{\fki}{\mathfrak{i}}
\newcommand{\tfki}{\widetilde{\mathfrak{i}}}

\newcommand{\tA}{\widetilde{A}}
\newcommand{\tB}{\widetilde{B}}
\newcommand{\tI}{\widetilde{I}}
\newcommand{\tW}{\widetilde{W}}
\newcommand{\tX}{\mathbf{\widetilde{X}}}

\newcommand{\degree}{\operatorname{deg}}
\newcommand{\CouplingTime}{T} 
\newcommand{\ProbR}[1]{\operatorname{\mathbb{P}}_R\left[#1\right]}
\newcommand{\StochEquiv}{\stackrel{\mathcal{D}}{=}}

\begin{document}
\maketitle

\begin{abstract}
{We show how to build an immersion coupling of a two-dimensional Brownian motion $(W_1, W_2)$ 
along with 
\(\binom{n}{2} + n= \tfrac12n(n+1)\). 
integrals of the form $\int W_1^iW_2^j \circ{\operatorname{d}}W_2$, where
$j=1,\ldots,n$ and $i=0, \ldots, n-j$ 
for some fixed $n$.
The resulting construction is applied to the study of couplings of certain hypoelliptic diffusions 
(driven by two-dimensional Brownian motion using polynomial vector fields).
This work follows up previous studies concerning coupling of Brownian stochastic areas and time integrals 
(Ben Arous, Cranston and Kendall (1995), Kendall and Price (2004), Kendall (2007), Kendall (2009), Kendall (2013), Banerjee and Kendall (2015), Banerjee, Gordina and Mariano (2016))
and is part of an ongoing research programme aimed at gaining a better understanding of when it is possible to couple not only diffusions
but also multiple selected integral functionals of the diffusions.}
\end{abstract}









\section{Introduction} \label{sec:intro}
A \emph{coupling} of two probability measures $\mu_1$ and $\mu_2$, 
defined on respective measure spaces $(\Omega_1, \mathcal{F}_1)$ and $(\Omega_2, \mathcal{F}_2)$, 
is a joint law $\mu$ defined on the product space $(\Omega_1 \times \Omega_2, \mathcal{F}_1 \times \mathcal{F}_2)$ 
whose marginals are $\mu_1$ and $\mu_2$.
A coupling of Markov processes $X$ and $\widetilde{X}$ is 
an \emph{immersion coupling} when the joint process
$$
\{(X(t+s), \widetilde{X}(t+s)): s \ge 0\} \text{ conditioned on } \mathcal{F}_t
$$
is again a coupling of the laws of $X$ and $\widetilde{X}$, but now starting from $(X(t), \widetilde{X}(t))$.
This is also called a co-adapted coupling \citep{Kendall-2007}, 
or faithful coupling \citep{Rosenthal-1997}, 
and is very closely related to the
near-equivalent notion of a Markovian coupling 
 \citep{BurdzyKendall-2000},
which additionally constrains the joint process $(X(t), \widetilde{X}(t))$ to be Markovian with respect to the filtration $\left(\mathcal{F}_t\right)_{t \ge 0}$. Immersion couplings are typically very much easier to describe than general couplings, 
since they may be specified in causal ways using, for example, stochastic calculus.

In the following we consider couplings of smooth elliptic diffusions with \(d\)-dimensional state space ${\Reals}^d$.
Specifically, for $d \ge 2$ and $1\leq k \leq d$, 
consider the following Stratonovich stochastic differential equation on ${\Reals}^d$:
\begin{align}\label{eq:SDE}
X(t)\quad=\quad
x + \int_0^tV_0(X(s)){\d}s + \sum_{i=1}^k\int_0^tV_i(X(s)) \circ {\d}W_i(s)\,.
\end{align}
Here $x \in {\Reals}^d$ is the initial state, 
$V_1, \dots,V_k$ are 
smooth vector fields,
and $(W_1,\dots,W_k)$ is a standard Brownian motion on ${\Reals}^k$.
We will consider couplings of two copies $X$ and $\widetilde{X}$ of this diffusion,
starting from arbitrary distinct initial states $x, \widetilde{x} \in {\Reals}^d$.
Interest focusses on
the \emph{coupling time}
$\CouplingTime$, 
defined as
$$
\CouplingTime\quad=\quad
\inf\{t \ge 0: X(s)= \widetilde{X}(s) \text{ for all } s \ge t\}\,.
$$
The coupling is said to be \emph{successful} if almost surely $\CouplingTime < \infty$ 
(where ``almost surely'' refers to the coupling measure $\mu$).
A major motivation to study couplings arises from the so-called \citeauthor{Aldous-1983}' coupling inequality
(see \citealp{Aldous-1983}):
\begin{equation}\label{eq:Aldous}
\mu(\tau>t) 
\quad\ge\quad
\|\mu_{t}-\widetilde{\mu}_{t}\|_{TV}\,, \qquad\text{ for all } t \ge 0\,,
\end{equation}
where $\mu_{t}$ and $\widetilde{\mu}_{t}$ denote the laws of $X(t)$ and $\widetilde{X}(t)$ respectively and $||\cdot||_{TV}$ denotes the total variation distance between probability measures given by
$$
\|\mu_1 - \mu_2\|_{TV} \quad=\quad
\sup\{|\mu_1(A)-\mu_2(A)| : \text{ measurable } A\}\,.
$$
Using inequality \eqref{eq:Aldous},
construction of a 
coupling of $X$ and $\widetilde{X}$ 
automatically 
bounds
the total variation distance between the laws of the diffusions at time $t$.
A \emph{maximal coupling}
is one for which 
the inequality \eqref{eq:Aldous} is actually an equality for all $t$. These have been shown to exist under very general conditions \citep{Griffeath-1975,Pitman-1976,Goldstein-1979,SverchkovSmirnov-1990,ErnstKendallRobertsRosenthal-2017}. However in most cases the task of explicitly constructing such a maximal coupling is extremely hard, if not impossible. This provides strong motivation for considering 
immersion
couplings which, although not maximal in most cases \citep{BanerjeeKendall-2014,Kuwada-2009}, are easier to describe and can provide helpful bounds via \eqref{eq:Aldous}. 

Immersion couplings have been extensively studied for \emph{elliptic} diffusions (which is to say
diffusions given by \eqref{eq:SDE} when $k=d$ and $\{V_1(x), \dots, V_d(x)\}$ form a basis for ${\Reals}^d$ at each $x \in {\Reals}^d$).
The simplest example of such a coupling is the
\emph{reflection coupling} of Euclidean Brownian motions starting from two different points:
the second Brownian path is obtained from the first by reflecting the first path on the
hyperplane bisecting the line joining the starting points until the first path (equivalently,
the second, reflected, path) hits this hyperplane.
This coupling turns out to be maximal as well as Markovian!
\cite{LindvallRogers-1986} extended this reflection construction to produce successful Markovian couplings for 
\emph{elliptic diffusions} on ${\Reals}^d$ with bounded and Lipschitz drift and diffusion coefficients when the diffusion matrix does not vary too much in space
(see also \citealp{ChenLi-1989}).
However in general the construction is not symmetric between the two coupled processes,
and this method is not easily applicable when the diffusion matrix varies appreciably over space.
A more geometric approach, depending symmetrically on the two coupled diffusions,
is provided by the \emph{Kendall-Cranston coupling} \citep{Kendall-1986,Cranston-1991}.
Consider the positive-definite diffusion matrix $\sigma(x)$ formed with columns $V_1(x),\dots,V_d(x)$.
As $x$ varies over ${\Reals}^d$, so this furnishes ${\Reals}^d$ with a Riemannian metric $g$ given by
$g(x)=(\sigma(x)\sigma(x)^\top)^{-1}$.
With this intrinsic \emph{diffusion metric}, ${\Reals}^d$ becomes a Riemannian manifold and the diffusion can be recognized as a Brownian motion with drift on this manifold.
One then uses an appropriate generalization of reflection involving parallel transport along geodesics to obtain reflection couplings.
These couplings are successful when (for example) the Ricci curvature of the manifold is non-negative and the drift vector field satisfies appropriate regularity conditions.
We note here that the Kendall-Cranston coupling also works for elliptic diffusions whose state space is any smooth manifold, 
and can be applied to even more general situations \citep{VonRenesse-2004}.

The above techniques fail for diffusions which are not elliptic, as there is no natural Riemannian metric intrinsic to such diffusions.
However, an important class of non-elliptic diffusions has attracted attention in recent times: namely, the \emph{hypoelliptic diffusions}.
These are diffusions $(X(t):t \ge 0)$ such that $X(t)$ has a \emph{smooth density with respect to Lebesgue measure} for each $t>0$.
They arise naturally in a variety of contexts: for example,
modelling the motion of a particle following Newton's equations under a potential, 
white noise random forcing and linear friction (the \emph{kinetic Fokker-Plank diffusion}, \citealp{Villani-2006}), 
describing stochastic oscillators (the \emph{Kolmogorov diffusion}, \citealp{MarkusWeerasinghe-1988}), 
quantum mechanics and rough paths theory (\emph{Brownian motion on the Heisenberg group},
\citealp{Neuenschwander-1996}, \citealp{FrizHairer-2014}) 
and modelling of macromolecular systems \citep{GrubmullerTavan-1994}.
All these examples place a premium on gaining a good understanding of the behaviour of hypoelliptic diffusions.
In particular, the construction of successful couplings for these diffusions immediately implies, \emph{via} Aldous' inequality \eqref{eq:Aldous}, 
that the total variation distance between the laws of two such diffusions started from distinct points converges to zero as time goes to infinity.
Furthermore, estimates on the coupling time distribution 
deliver
bounds on the convergence rate.
This, in turn, yields estimates of rate of convergence to stationarity, when a stationary measure exists.
Moreover, these couplings can also be used to furnish gradient estimates for harmonic functions corresponding to the generators of the diffusions \emph{via} purely probabilistic means 
\citep{Cranston-1992,Cranston-1991,BanerjeeGordinaMariano-2016}.

At the time of writing, coupling of hypoelliptic diffusions have only been studied for rather specific examples.
Hypoelliptic diffusions can be viewed as ``high dimensional processes driven by low dimensional Brownian motions'', 
which suggests that the goal of producing successful Markovian couplings of such diffusions
may be best achieved by learning how 
to produce Markovian couplings of the driving Brownian motion together with a (typically finite) collection of path functionals.
These couplings, sometimes described as \emph{exotic couplings}, were first studied in \cite{BenArousCranstonKendall-1995}. 
This described successful Markovian couplings for the \emph{Kolmogorov diffusion of order one} 
(given by a Brownian motion $B$ along with its running time integral $\int_0^tB(s){\d}s$) and \emph{Brownian motion on the Heisenberg group}
(a two-dimensional Brownian motion $(B_1, B_2)$ together with its L\'evy stochastic area $\int_0^tB_1(s){\d}B_2(s) - \int_0^tB_2(s){\d}B_1(s)$).
\cite{KendallPrice-2004} 
showed how to generate
successful Markovian couplings for the Kolmogorov diffusion of any finite order $n$ 
(a Brownian motion $B$ along with its $n-1$ iterated time integrals 
$\int\cdots\int_{0\le s_1 \le \dots \le s_i \le t}B(s_1){\d}s_1{\d}s_2\dots {\d}s_i$ for $1\le i \le n-1$).
Later 
\cite{Kendall-2007,Kendall-2009d} described a construction of a successful Markovian coupling of Brownian motion 
on the 
step-two free nilpotent Lie
group of any underlying finite dimension $n$ (corresponding to
an \(n\)-dimensional Brownian motion $(B_1,\dots, B_n)$ together with the 
\(\binom{n}{2}\)
stochastic areas $\int_0^tB_i(s){\d}B_j(s) - \int_0^tB_j(s){\d}B_i(s)$ for $1 \le i < j \le n$, or, 
using vector notation, \(\underline{B}\) together with the alternating vector-product
\(\int \underline{B}\wedge\d\underline{B}\)).
\cite{Kendall-2013a} described how to couple scalar Brownian motion together with local time, and used this to couple a rather degenerate diffusion arising in stochastic control theory.
Even in these rather simple examples, the coupling constructions turn out to be quite complicated.
Simpler cases use careful combinations of reflection coupling and \emph{synchronous coupling} (making Brownian increments agree):
\cite{Kendall-2007,Kendall-2009d} show that one also needs to use more varieties of coupling
(for example what might be called ``rotation couplings'')
when coupling all stochastic areas for Brownian motion in dimension of \(3\) or greater.

In this article, we provide constructions of immersion (in fact, Markovian) couplings for a considerable range of diffusions of the form given in \eqref{eq:SDE} with \(k=2\),
based on polynomial vector fields \(V_i\).
This is a significant step beyond \cite{Kendall-2007,Kendall-2009d} 
in the development of the programme of understanding coupling for hypoelliptic diffusions,
albeit limited here to the case of an underlying two-dimensional Brownian motion.
Before going into the detailed description of the problem, we define the \emph{parabolic H{\"o}rmander condition} which will be a crucial assumption in the coupling construction.

Consider the following sets of vector fields: 
\begin{align*}
\mathcal{V}_0 \quad=\quad\{V_i: i \ge 1\}, \hspace{1.5cm}
\mathcal{V}_{j+1} \quad=\quad\{[U,V_i]: U \in \mathcal{V}_j, \ i \ge 0\} \ \text{for } j \ge 0,
\end{align*}
where $[U,V]$ denotes the Lie bracket of the vector fields $U$ and $V$.
Set $\mathcal{\mathbf{V}}_j(x)= \operatorname{span}\{V(x), \ V \in \mathcal{V}_j\}$.
We will make the following assumption:
\begin{itemize}
\item[(PHC)] The vector fields $V_0, V_1, \dots V_k$ satisfy the \emph{parabolic H{\"o}rmander condition}, i.e., $\bigcup_{j \ge 0} \mathcal{\mathbf{V}}_j(x)= {\Reals}^d$ for each $x \in {\Reals}^d$.
\end{itemize}
Subject 
to 
suitable
regularity conditions, (PHC) is a \emph{necessary} assumption if we 
want to construct successful couplings from arbitrary pairs of starting points.
To see this, consider the distribution of sub-spaces \(\{\mathcal{D}(x)=\bigcup_{k \ge 0} \mathcal{\mathbf{V}}_j(x), \ x \in \mathbb{R}^d\}\) generated by $\bigcup_{j \ge 0}\mathcal{V}_j$ 
(that is, the smoothly varying subspace of the tangent space spanned by these vector fields).
\cite{Lobry-1970}
showed that if $\mathcal{D}$ is of ``locally finite type'' (in particular, if the vector fields $V_0, V_1,\dots, V_k$ are real analytic), then it has the \emph{maximal integral manifold property}, i.e., for each point $x \in {\Reals}^d$, there exists an immersed submanifold $S$ (called an integral manifold) containing $x$ with the property that its tangent bundle coincides with the distribution $\mathcal{D}$.
Moreover, 
$S$ 
can be chosen so
that any other integral manifold which intersects $S$ must be an open submanifold of $S$ (note that $S$ need not be complete in $\mathbb{R}^d$).
In this case, ${\Reals}^d$ splits into disjoint maximal integral manifolds.
It follows from support theorems \citep[for example]{IkedaWatanabe-1981} that if 
a diffusion starts
from a point inside one maximal integral manifold 
then almost surely it must stay in this manifold for all time.
Thus, under regularity conditions such as real analyticity, if (PHC) does not hold, then there must be at least two disjoint maximal integral manifolds.
Consequently, two copies of the diffusion started from points in different maximal integral manifolds will almost surely never meet.

In order to make progress towards answering the general question of whether it is possible to construct successful immersion couplings of 
a diffusion satisfying (PHC) from arbitrary pairs of distinct starting points,
this article considers a simplification.
It will be convenient to view \(\Reals^d=\Reals\times\Reals\times\Reals^{d-2}\),
with corresponding coordinates $w=(w_1, w_2, w_3) \in {\Reals} \times {\Reals} \times {\Reals}^{d-2}$
\emph{etc}
(in a mild abuse of notation, \(w_3\) denotes a \((d-2)\)-dimensional vector).
We assume that our diffusions satisfy \eqref{eq:SDE} when the \emph{drift vector field $V_0=0$}, 
the \emph{driving Brownian motion is two-dimensional} (i.e.
$k=2$) and the \emph{driving vector fields $V_1$ and $V_2$ are polynomial functions of the driving Brownian motion}.
Specifically,
for $d \ge 3$ and for each $w=(w_1, w_2, w_3) \in {\Reals} \times {\Reals} \times {\Reals}^{d-2}$, 
suppose that $X$ can be written as
\begin{equation}\label{eq:underlying-sde}
X(t)\quad=\quad
(w_1+W_1(t), w_2 + W_2(t), X_3(t))
\end{equation}
where $(W_1, W_2)$ is a two-dimensional standard Brownian motion and $X_3$ 
can be written in vector format as satisfying the Stratonovich differential equation
\begin{equation}\label{eq:diffform}
X_3(t)\quad=\quad
w_3 + \sum_{i=1}^2\int_0^t \sigma_i(w_1 + W_1(s),w_2 + W_2(s)) \circ {\d}W_i(s).
\end{equation}
Here $\sigma_1$, $\sigma_2$ are (\((d-2)\)-dimensional) vector-valued \emph{polynomials}:
\begin{align*}
\sigma_i(x_1, x_2) 
\quad=\quad
a_i^{0,0} + \sum_{1 \le l+m \le n} a_i^{l,m} x_1^lx_2^m, \qquad\text{ for } i=1, 2\,,
\end{align*}
with \((d-2)\)-dimensional vector-valued coefficients  $a_i^{l,m}=(a_{i,3}^{l,m}, \dots, a_{i,d}^{l,m})^\top \in {\Reals}^{d-2}$.
For convenience, write
\[
\sigma_i(x_1, x_2) \quad=\quad (\sigma_{i,3}(x_1,x_2), \dots, \sigma_{i,d}(x_1,x_2))^\top\,. 
\]
Lemma \ref{lem:PHC} below describes exactly when the system \eqref{eq:diffform} satisfies (PHC).

Several important examples of hypoelliptic diffusions fall in this category, including \emph{Brownian motion on the Heisenberg group} \citep{Neuenschwander-1996,BenArousCranstonKendall-1995}.
The problem of immersion coupling for diffusions in the form of \eqref{eq:diffform} makes a useful next step in the bigger program of coupling hypoelliptic diffusions because of the following reasons.
Firstly, the zero drift condition helps to simplify (PHC) and give a clearer exposition, 
although we believe that the methods developed here can be used even when the drift is non-zero
but satisfies certain growth conditions.
Secondly, the polynomial form of the driving vector fields ensures that $X_3$ can be written 
using linear combinations of monomial Stratonovich integrals of the form $(\int W_1^iW_2^j \circ{\d}W_2: {i +j \le n})$ and thus, the problem reduces to successfully coupling the driving Brownian motions along with these integrals.
Moreover, these polynomial vector fields can be used to approximate a large class of real analytic and nilpotent vector fields and we hope that our technique will extend to more general diffusions driven by such vector fields.
Thirdly, as proved in Lemma \ref{lem:PHC} below, (PHC) for this class of diffusions simplifies to a non-singularity condition for a matrix formed by the vectors $a_i^{l,m}$.
Finally, as described in \cite{Kendall-2007} in the simpler context of coupling stochastic areas, 
successful Markovian coupling strategies can be achieved using only reflection/synchronous coupling of Brownian motions when the driving Brownian motion is two-dimensional 
(for example, Brownian motion on the Heisenberg group), 
but for higher dimensional analogues it is necessary to employ rotation couplings (which is to say, control strategies using orthogonal matrices),
and this complicates the coupling strategy considerably.
As we will see,
the restriction to a two-dimensional driving Brownian motion 
in the case of \eqref{eq:diffform} similarly
allows for a rather explicit coupling construction using only synchronous coupling of $W_2$ at all times together with judicious switching between synchronous and reflection phases for $W_1$. However, we anticipate that one of the challenges of dealing with higher-dimensional Brownian motions will be to deal with
complexity entailed by no longer being able to keep one coordinate synchronously coupled and in agreement for all time.
We 
plan to
address the complexities of the higher dimensional case in a subsequent article.

In the remainder of this section, we will show that,
in order to successfully couple two copies $X$ and $\widetilde{X}$ of our diffusion
\eqref{eq:underlying-sde} started from distinct points, 
it suffices successfully to couple simultaneously the driving Brownian motions along with integrals of the form 
$(\int W_1^iW_2^j \circ{\d}W_2 : {i +j \le n})$.
Define the (\((d-2)\)-dimensional) vector-valued function
\begin{equation}\label{eq:def-of-phi}
\phi(x_1, x_2)
\quad=\quad 
\sigma_2(x_1,x_2)-\int_{w_1}^{x_1}\partial_2 \sigma_1(u, x_2) {\d}u 
\end{equation}
where $(x_1,x_2) \in {\Reals}^2$ and $\partial_i$ denotes 
the partial derivative with respect to the \(i^\text{th}\) coordinate $(i=1,2)$.

Set $\Psi_1(x_1,x_2)= \int_{w_1}^{x_1}\sigma_1(u,x_2){\d}u$.
Computing the Stratonovich differential of \(\Psi_1(w_1 + W_1(t),w_2 + W_2(t))\),
and then integrating this differential,
amounts to establishing an integration-by-parts relation between
certain Stratonovich integrals with respect to \(W_1\)
and other Stratonovich integrals with respect to \(W_2\), 
holding up to addition of 
a function of \(W_1\) and \(W_2\) whose coupling follows directly from coupling of \((W_1,W_2)\): 
\begin{multline*}
\int_0^t\sigma_1(w_1 + W_1(s),w_2 + W_2(s)) \circ {\d}W_1(s)
\quad=\quad \\
\Psi_1(w_1 + W_1(t),w_2 + W_2(t))
 - \int_0^t \partial_2\Psi_1(w_1 + W_1(s),w_2 + W_2(s)) \circ {\d}W_2(s)\,.
\end{multline*}
Hence \(X_3\) can be expressed as the sum of a function of \(W_1\) and \(W_2\) and a Stratonovich integral
with respect to \(W_2\) alone:
\begin{align}\label{eq:psi}
X_3(t)& \quad=\quad
 w_3 + \int_0^t\sigma_1(w_1 + W_1(s),w_2 + W_2(s)) \circ {\d}W_1(s)\nonumber\\
&\quad\qquad + \int_0^t\sigma_2(w_1 + W_1(s),w_2 + W_2(s)) \circ {\d}W_2(s)
\nonumber\\
&\quad=\quad w_3 + \Psi_1(w_1 + W_1(t),w_2 + W_2(t))
\nonumber\\
&\quad\qquad  + \int_0^t \left[\sigma_2(w_1 + W_1(s),w_2 + W_2(s)) - \partial_2\Psi_1(w_1 + W_1(s),w_2 + W_2(s))\right] \circ {\d}W_2(s)
\nonumber\\
&\quad=\quad w_3 + \Psi_1(w_1 + W_1(t),w_2 + W_2(t)) + \int_0^t \phi(w_1 + W_1(s),w_2 + W_2(s)) \circ {\d}W_2(s)\,.
\end{align}

Let $\Sigma(x_1,x_2)$ denote the $(d-2) \times \frac{n(n-1)}{2}$ matrix formed by 
arranging in a row the \(\tfrac{n(n-1)}{2}\) different \((d-2)\)-dimensional vector-valued functions
$\partial_1^{l+1}\partial_2^{m}\phi(x_1,x_2)$ (here ${1 \le l+1+m \le n}$).
The condition (PHC) is equivalent to a rank condition on the matrix-valued function \(\Sigma\):
\begin{lem}\label{lem:PHC}
The diffusion $X$ satisfies (PHC) if and only if $\Sigma(w_1,w_2)$ has full rank $d-2$, 
where $w$ is the starting point of the diffusion in question.
\end{lem}
\begin{proof}
Let \(\partial_1\), \(\partial_2\), \ldots,  \(\partial_d\) represent the standard basis vectors of \(\Reals^d\).
The diffusion $X$ can be expressed in the form
$$
X(t)\quad=\quad w + 
\int_0^tV_1(X(s)) \circ {\d}W_1(s)
+
\int_0^tV_2(X(s)) \circ {\d}W_2(s)
$$
where $w=(w_1,w_2,w_3)$ and the vector fields $V_1, V_2$ are given by
$$
V_i(x) \quad=\quad
\partial_i + \sum_{j=3}^d \sigma_{i,j}(x_1,x_2) \partial_j
$$
for $i=1,2$ and $x=(x_1, x_2, x_3) \in \Reals\times\Reals\times{\Reals}^{d-2}$ 
(so \(x_3\) is a \((d-2)\)-dimensional vector).
For any word $I=(i_1,\dots, i_N) \in \{1,2\}^N$, write $V_I=[V_{i_1}[V_{i_2}[\dots [V_{i_{N-1}},V_{i_N}]]\dots]$.
Denote by $n_1(I)$, $n_2(I)$ the number of occurrences of $1, 2$ respectively in the word $I$.
We claim that if \(I\) is of length 2 or more and \(i_{N-1}=1\), \(i_N=2\) then
\begin{equation}\label{eq:hor}
V_I(x)\quad=\quad
\sum_{j=3}^d\partial_1^{n_1(I)}\partial_2^{n_2(I)-1}\phi_j(x_1,x_2)\partial_j\,.
\end{equation}
We prove this by induction on the length of the word \(I\).
For length 2, \eqref{eq:hor} follows from the definition of \(\phi\) in \eqref{eq:def-of-phi}:
\begin{align*}
[V_1,V_2] (x) &\quad=\quad
\sum_{j=3}^d
\left(
   \partial_1\sigma_{2,j}(x_1,x_2) - 
   \partial_2\sigma_{1,j}(x_1,x_2)
\right)\partial_j
   \quad=\quad
   \sum_{j=3}^d \partial_1\phi_j(x_1,x_2)\partial_j\,.
\end{align*}
For any $N >2$ assume that the induction hypothesis holds for words of length less than $N$.
Consider \(I\in\{1,2\}^{N-1}\) with \(i_{N-2}=1\), \(i_{N-1}=2\), and examine the case of \(I^*=(i_0,I)\).
By the induction hypothesis,
$$
V_{I}(x)\quad=\quad
\sum_{j=3}^d\partial_1^{n_1(I)}\partial_2^{n_2(I)-1}\phi_j(x_1,x_2)\partial_j\,.
$$
Observe that $\sigma_1, \sigma_2$ depend only on $x_1, x_2$, 
so that the form of \(V_I\) implies that \(V_{I} V_{i_0}=0\)
and \(V_{i_0} V_{I} =\partial_{i_0}V_{I}\). Therefore
\begin{align*}
V_{I^*}(x) &\quad=\quad
\partial_{i_0}V_{I}(x)
\quad=\quad 
\sum_{j=3}^d\partial_{i_0}\partial_1^{n_1(I)}\partial_2^{n_2(I)-1}\phi_j(x_1,x_2)\partial_j\\
&\quad=\quad
\sum_{j=3}^d\partial_1^{n_1(I^*)}\partial_2^{n_2(I^*)-1}\phi_j(x_1,x_2)\partial_j
\end{align*}
proving \eqref{eq:hor} in that case also.
Note that the coefficients of $\partial_1$ and $\partial_2$ are zero in $V_{i_1,\dots, i_{N-2},1,2}$.

Now observe that $V_1(x)$ and $V_2(x)$ are linearly independent for each $x \in {\Reals}^d$.
Furthermore, 
\(V_{i_1,\dots, i_{N-2},2,1}=-V_{i_1,\dots, i_{N-2},1,2}\)
while \(V_{i_1,\dots, i_{N-2},1,1}=V_{i_1,\dots, i_{N-2},2,2}=0\),
so
the coefficients of $\partial_1$ and $\partial_2$ for $V_I$ are zero whenever the word $I$ has length greater than or equal to $2$.
Thus, for (PHC) to hold, the subspace spanned by $\{V_I: \text{ length}(I) \ge 2\}$ must have dimension $d-2$.
By \eqref{eq:hor} and the definition of $\Sigma(x_1,x_2)$, this is equivalent to requiring that $\Sigma(x_1,x_2)$ has rank $d-2$.
But $\phi$ is a vector polynomial, 
so $\Sigma(x_1,x_2)$ has rank $d-2$ if and only if $\Sigma(w_1,w_2)$ has rank $d-2$.
To see this, note that if $\Sigma(x_1,x_2)$ has rank $d-2$ for some $x \in \Reals^d$, then there exists a $(d-2) \times (d-2)$ sub-matrix $\Sigma^*(x_1, x_2)$ of $\Sigma(x_1,x_2)$ which is non-singular. From the continuity of the determinant, we conclude that $\Sigma^*(y_1, y_2)$, and hence $\Sigma(y_1, y_2)$, has rank $d-2$ for all $y$ in an open neighborhood of $x$. Conversely, if $\Sigma(x_1,x_2)$ has rank less than $d-2$ for some $x \in \Reals^d$, then there exist constants $c_1,\dots, c_d$, not all zero, such that
$$
\sum_{k=3}^d c_k \partial_1^{l+1}\partial_2^{m}\phi_k(x_1,x_2) 
\quad=\quad
0\,, 
\qquad \text{ for } 1 \le l+1+m \le n\,.
$$
As $\phi(x_1,x_2)=(\phi_3(x_1, x_2),\dots,\phi_d(x_1,x_2))$ is a vector polynomial in $x_1, x_2$ of degree $n$,
$$
\partial_1\phi(y_1,y_2)
\quad=\quad
\sum_{0 \le l+m \le n-1} \frac{\partial_1^{l+1}\partial_2^{m}\phi(x_1,x_2)}{l! \ m!}(y_1-x_1)^l(y_2 - x_2)^m\,,
 \qquad\text{ for } y \in \Reals^d\,.
$$
Hence,
$$
\sum_{k=3}^d c_k \partial_1\phi_k(y_1,y_2) 
\quad=\quad
\sum_{0 \le l+m \le n-1} \sum_{k=3}^d c_k\frac{\partial_1^{l+1}\partial_2^{m}\phi_k(x_1,x_2)}{l! \ m!}(y_1-x_1)^l(y_2 - x_2)^m
\quad=\quad0
\,.
$$
So further differentiation yields 
$$
\sum_{k=3}^d c_k \partial_1^{l+1}\partial_2^{m}\phi_k(y_1,y_2) 
\quad=\quad0
\,, 
\qquad\text{ for } 1 \le l+1+m \le n,  \ \ y \in \Reals^d\,.
$$
Thus, $\Sigma(y_1,y_2)$ has rank less than $d-2$ for all $y \in \Reals^d$. From the connectedness of $\Reals^d$, $\Sigma(x_1,x_2)$ has rank $d-2$ for some $x \in \Reals^d$ if and only if $\Sigma(w_1,w_2)$ has rank $d-2$, completing the proof of the lemma.
\end{proof}
\begin{rem}
It follows from the reasoning in the proof of Lemma \ref{lem:PHC} 
that $V_I=0$ whenever the length of the word $I$ strictly exceeds $n$ 
(the maximal degree of the polynomial coefficients \(\sigma_1\), \(\sigma_2\)), 
regardless of whether (PHC) is satisfied or not.
From this observation it follows that any diffusion 
of the form implied by \eqref{eq:diffform} is \emph{nilpotent} \citep{Baudoin-2004}.
Nilpotent diffusions serve as the starting point for many analyses of hypoelliptic diffusions owing to their simplicity.
\end{rem}
Consider the task of immersion coupling two copies $X$ and $\widetilde{X}$ of the diffusion starting from $w$ and $\widetilde{w}$ respectively,
using driving Brownian motions $(W_1, W_2)$ and $(\widetilde{W}_1, \widetilde{W}_2)$.
Reflection coupling can be used to bring together the driving Brownian motions first.
Thus there is no loss of generality in assuming that the \emph{Brownian} starting points agree: $(w_1, w_2)=(\widetilde{w}_1, \widetilde{w}_2)$.
Referring to the representation \eqref{eq:psi}, it suffices to couple the two diffusions
\begin{align*}
X^*(t)\quad&=\quad
\left(w_1 + W_1(t),w_2 + W_2(t), 
 w_3 + \int_0^t \phi(w_1 + W_1(s),w_2 + W_2(s)) \circ {\d}W_2(s)\right)\,,
\\
\widetilde{X}^*(t)\quad&=\quad
\left(w_1 + \widetilde{W}_1(t),w_2 + \widetilde{W}_2(t),
 \widetilde{w}_3 + \int_0^t \phi(w_1 + \widetilde{W}_1(s),w_2 + \widetilde{W}_2(s)) \circ d\widetilde{W}_2(s)\right)\,,
\end{align*}
with starting points differing only in the third, vectorial, coordinates \({w}_3\), \(\widetilde{w}_3\).
(This is because coupling of the summands \(\Psi_1(w_1 + W_1(t),w_2 + W_2(t))\) and \(\Psi_1(w_1 + \widetilde{W}_1(t),w_2 + \widetilde{W}_2(t))\) in \eqref{eq:psi} 
is immediately implied by coupling of \((W_1,W_2)\) and \((\widetilde{W}_1, \widetilde{W}_2)\).)
Denote by $\mathcal{I}(t)$ the vector formed by 
$\left(\frac{\int_0^t{W_1(s)^{l+1}}{W_2(s)^m} \circ {\d}W_2(s)}{(l+1)! \ m!}:{1 \le l+1+m \le n}\right)$ and similarly $\widetilde{\mathcal{I}}(t)$.
Decomposing the polynomials given by the integrands \(\phi(w_1 + W_1(s),w_2 + W_2(s))\) and \(\phi(w_1 + \widetilde{W}_1(s),w_2 + \widetilde{W}_2(s))\)
according to whether or not monomials involve \(W_1\) 
(respectively \(\widetilde{W}_1\)),
the last $d-2$ coordinates of $X^*$, $\widetilde{X}^*$ can be written in vector form as $X^*_3$, $\widetilde{X}^*_3$ where
\begin{align*}
X^*_3(t)&\quad=\quad
w_3 + P(w_2, w_2+W_2(t)) + \Sigma(w_1,w_2) \mathcal{I}(t)\,,
\\
\widetilde{X}^*_3(t)&\quad=\quad
\widetilde{w}_3 + P(w_2, w_2+ \widetilde{W}_2(t)) + \Sigma(w_1,w_2) \widetilde{\mathcal{I}}(t)\,,
\end{align*}
where $P$ is a polynomial that arises from Stratonovich integration (with respect to $W_2$) of monomials in $w_2 + W_2$ alone.
By Lemma \ref{lem:PHC}, (PHC) implies that $\Sigma(w_1,w_2)$ has rank $d-2$, 
hence $w_3, \widetilde{w}_3$ both lie in the space spanned by the columns of $\Sigma(w_1,w_2)$.
Thus, there are $z_3^*, \widetilde{z}_3^* \in {\Reals}^{n(n-1)/2}$ 
such that $w_3 = \Sigma(w_1,w_2) z_3^*$, \(\widetilde{w}_3 = \Sigma(w_1,w_2) \widetilde{z}_3^*\).
Hence
\begin{align*}
X^*_3(t)\quad&=\quad
P(w_2, w_2+W_2(t)) + \Sigma(w_1,w_2) (z_3^* + \mathcal{I}(t))\,,
\\
\widetilde{X}^*_3(t)\quad&=\quad
P(w_2, w_2+\widetilde{W}_2(t)) + \Sigma(w_1,w_2) (\widetilde{z}_3^* + \widetilde{\mathcal{I}}(t))\,.
\end{align*}
It follows that if we can successfully couple
\begin{align}
&(w_1 + W_1(t), w_2 + W_2(t), z_3^* + \mathcal{I}(t))\,,
\nonumber\\
&(\widetilde{w}_1 + \widetilde{W}_1(t), \widetilde{w}_2 + \widetilde{W}_2(t), \widetilde{z}_3^* + \widetilde{\mathcal{I}}(t))
\label{eq:reduction}
\end{align}
from arbitrary pairs of starting points $(w_1, w_2, z_3^*), (\widetilde{w}_1, \widetilde{w}_2, \widetilde{z}_3^*) \in {\Reals} \times {\Reals} \times  {\Reals}^{n(n-1)/2}$, 
then we can successfully couple
$X$ and $\widetilde{X}$ from arbitrary pairs of starting points.
In the next section, we construct a coupling of the above processes.
We show in Theorem \ref{thm:couplemain} that this coupling is indeed successful, and that
moreover the coupling time has a power law tail.
The existence of such a successful coupling immediately implies the following theorem.
\begin{thm}\label{thm:hypocouple}
Consider the diffusion $X(t)=(X_1(t), X_2(t), X_3(t)) \in {\Reals}^d$ (for $d \ge 3$, considering \(X_3\) as a \((d-2)\)-dimensional process) defined by the following stochastic differential equation:
\begin{align*}
{\d}X_1(t)\quad&=\quad{\d}W_1(t)\,,
\\
{\d}X_2(t)\quad&=\quad{\d}W_2(t)\,,
\\
{\d}X_3(t)\quad&=\quad
\sigma_1(X_1(t),X_2(t)) \circ {\d}W_1(t)
+
\sigma_2(X_1(t),X_2(t)) \circ {\d}W_2(t)
\,,
\end{align*}
where $(W_1, W_2)$ is a two-dimensional standard Brownian motion and $\sigma_1, \sigma_2$ are polynomial vector fields such that (PHC) holds.
Then there exists a successful Markovian coupling of two copies of the above diffusion starting from any pair of distinct points.
\end{thm}
Theorem \ref{thm:hypocouple} summarizes the qualitative content of (and is a direct consequence of) 
Theorem \ref{thm:couplemain} stated and proved in Section \ref{sec:coup-multiple} below,
but omits the tail estimate on the coupling time distribution.
It is stated here as a separate theorem in order to highlight how the Brownian integral couplings constructed in the subsequent sections 
connect to the general theme of coupling hypoelliptic diffusions.

\section{Technical preliminaries}\label{sec:TP}
To facilitate inductive arguments in the following proofs,
we 
fix 
a total ordering \(\preceq\) of the discrete simplex 
\(\Delta_n=\{(a,b) \in \Integers^2: 0 \le a, b,  a + b \le n\}\) (for some fixed \(n\geq1\)).
We achieve this by specifying a function \(f: \Delta_n \rightarrow \Integers\) and
defining the order by  \((a,b) \preceq (c,d)$ if $f(a,b) \le f(c,d)\).

We choose \(f(a,b)=2na + (2n+1)b\): totality of \(\preceq\) follows since \(f\) takes values in the totally ordered set \(\Integers\);
antisymmetry holds by a parity argument showing that \(f(a,b)=f(c,d)\) if and only if \(a=c\) and \(b=d\); transitivity is immediate.
Note that the \(\preceq\)-maximal element of \(\Delta\) is \((0,n)\).
We remark that \(\preceq\) can be replaced by any other total ordering extending the partial ordering induced by 
considering \(a+b\).


From here onwards, to save cumbersome notation, $(W_1,W_2)$ will denote a two-dimensional Brownian motion starting from a general point $(W_1(0), W_2(0)) \in \mathbb{R}^2$.
Let $(a,b) \in \Delta_n$ be the index representing the Brownian Stratonovich integral
\(I_{(a,b)}(t)= I_{(a,b)}(0) + \int_0^tW_1^aW_2^b\circ {\d}W_2\) (so \(I_{(0,0)}(t)=W_2\)).
We shall refer to such integrals as \emph{monomial Stratonovich integrals}.
Consider the \(\preceq\)-ordered collection of Brownian integrals (deeming \(W_1\) to have precedence over all \(I_{(i,j)}\) for \((i,j)\in\Delta\))
$$
\mathbf{X}_{(a,b)}
\quad=\quad
\left(W_1, I_{(c,d)} ; {(c,d) \preceq (a,b)}, \ c \ge 1\right)\,.
$$
In the following, it is only necessary to consider 
$c \ge 1$;
in the case
$c=0$, $I_{(c,d)}$ reduces to a monomial in $W_2$ and so $W_2$ and its coupled counterpart will take equal values for all time in our coupling construction.
Notice also the following: if $a=0$ and $(a^-,b^-)$ is the predecessor of $(a,b)=(0,b)$ in the $\preceq$ ordering, then $\mathbf{X}_{(a^-,b^-)}=\mathbf{X}_{(a,b)}$. This is because
$
\{(c,d) \preceq (0,b), c \ge 1\}
=\{(c,d) \prec (0,b), c \ge 1\} = \{(c,d) \preceq (a^-,b^-), c \ge 1\}.
$

%

Scaling arguments play a major r\^ole in the study of these couplings.
The following lemma records a simple but crucial fact about scaling for Stratonovich integrals of Brownian motions.
Consider the scaling transform \(\mathcal{S}_r\), defined for any scalar \(r\) by 
$$
\mathcal{S}_r\left(\X_{(a,b)}\right)
\quad=\quad
\left(r W_1, r^{i+j+1}I_{(i,j)} : {(i,j) \preceq (a,b)}, i \ge 1\right)\,.
$$
Further, define
$
W_i^{(r)}(t) = rW_i(0) + (W_i(t) - W_i(0))
$  
for $i=1,2$ and
$
I^{(r)}_{(a,b)}(t) = r^{a+b+1}I_{(a,b)}(0) + \int_0^t(W_1^{(r)})^a(W_2^{(r)})^b\circ {\d}W_2
$
for $0 \le a, b,  a + b \le n$. Write
$$
\X^{(r)}_{(a,b)}
\quad=\quad
\left(W^{(r)}_1, I^{(r)}_{(c,d)} ; {(c,d) \preceq (a,b)}, \ c \ge 1\right)\,.
$$
\begin{lem}\label{lem:scaling}
The following distributional equality holds:
$$
\left(
\mathcal{S}_r\left(\X_{(a,b)}\right)(t) : t \ge 0
\right) 
\quad{\StochEquiv}\quad
\left( 
\X^{(r)}_{(a,b)}(r^2 t) : t \ge 0
\right) \,.
$$
\end{lem}

\begin{proof}
This is a direct consequence of linearity of Stratonovich integration taken together with the Brownian scaling property
$$
\left(rW_i(t): t \ge 0\right) 
\quad{\StochEquiv}\quad
\left(W_i^{(r)}(r^2t): t \ge 0\right)
\qquad \text{ for } i=1,2\,.
$$
\end{proof}
Note that \(\preceq\) is a total ordering extension
of the partial order on \(\Delta_1\) given by the
\emph{scaling degree}, \(\degree((a,b))=\degree({I}_{(a,b)})=a+b+1\).
Monomial Stratonovich integrals of lower degree evolve in time faster than those of higher degree;
this is a key reason why our inductive arguments will work.

\medskip

The following two technical lemmas complete the list of technical preliminaries.
\begin{lem}\label{lem:tail1}
Let $B_t$ be a standard Brownian motion adapted to a filtration $(\mathcal{F}_t:{t \ge 0})$.
Let $Y_t$ be a random process and let $\tau$ be a stopping time, both adapted to the same filtration.
\begin{itemize}
\item[(i)] Suppose there exists $\epsilon>0$
and some constants $C$, $\alpha>0, \beta>0$, not depending on $\eps$, 
such that for all \(M \ge 1\) and all \(t\geq\eps\),
\begin{align}
\Prob{\sup_{t \le \tau}|Y_t| \ge M} 
\quad\le\quad& CM^{-\alpha}\,,
\nonumber\\
\Prob{\tau \ge t} 
\quad\le\quad &C \; \left(\frac{\eps}{t}\right)^{\beta}\,.
\label{eq:eq:tail1-condition}
\end{align}
Then there exists a further constant $C'$ not depending on $\eps$,
and positive indices \(\gamma'\), \(\gamma''\) depending only on \(\alpha\), \(\beta\)  such that
\begin{align}
\Prob{\sup_{t \le \tau}\left|\int_0^tY_s{\d}B_s\right| \ge x} 
\quad&\le\quad
C' \;
\left(\frac{\eps}{x^4}\right)^{\gamma'}
\qquad
\text{ for }x \ge \eps ^{1/4}\,,
\nonumber\\
\Prob{\sup_{t \le \tau}\left|\int_0^tY_s{\d}s\right| \ge x} 
\quad&\le\quad
C' \;
\left(\frac{\eps}{x^2}\right)^{\gamma''}
\qquad
\text{ for }x \ge \eps ^{1/2}\,.
\label{eq:eq:tail1-result}
\end{align}
Here we may take \(\gamma'=(\alpha \wedge \beta \wedge 1)/8\) and \(\gamma''=(\alpha \wedge \beta \wedge 2)/8\).
\item[(ii)] Suppose there exists $\epsilon>0$
and some constants $C$, $\alpha>0, \beta>0$, not depending on $\eps$, such that for all $M \ge \eps$ and all \(t\geq1\),
\begin{align}
\Prob{\sup_{t \le \tau}|Y_t| \ge M} 
\quad&\le\quad C \; \left(\frac{\eps}{M}\right)^{\alpha}\,,
\nonumber\\
\Prob{\tau \ge t} 
\quad&\le\quad C \; t^{-\beta}\,.
\label{eq:eq:tail2-condition}
\end{align}
Then there exists a further constant $C'$ not depending on $\eps$,
and positive indices \(\gamma'\), \(\gamma''\) depending only on \(\alpha\), \(\beta\) such that
\begin{align}
\Prob{\sup_{t \le \tau}\left|\int_0^tY_s{\d}B_s\right| \ge x} 
\quad&\le\quad C' \; \left(\frac{\eps}{x}\right)^{\gamma'}
\text{ for }x \ge \eps\,,
\nonumber\\
\Prob{\sup_{t \le \tau}\left|\int_0^tY_s{\d}s\right| \ge x} 
\quad&\le\quad C' \; \left(\frac{\eps}{x^2}\right)^{\gamma''}
\text{ for }x \ge \eps^{1/2}\,.
\label{eq:eq:tail2-result}
\end{align}
Here we may take \(\gamma'=(\alpha \wedge \beta \wedge 1)/2\) and \(\gamma''=(\alpha \wedge \beta \wedge 2)/8\).
\end{itemize}
\end{lem}
\begin{proof}
Consider the stopping time
$$
\sigma_M \quad=\quad \inf\{t >0: |Y_t| \ge M\}\,.
$$
The Burholder-Davis-Gundy (BDG) inequality (see for example \citealp[p.~163]{KaratzasShreve-2012}),
respectively the monotonicity of the Lebesgue integral,
implies that, for any $M,T >0$, there exists a constant $C''>0$ not depending on $M,T$ such that
\begin{align*}
\Expect{\sup_{t \le T \wedge \sigma_M}\left|\int_0^tY_s{\d}B_s\right| ^2} 
\quad&\le\quad
C''\Expect{\int_0^{T \wedge \sigma_M}(Y_s)^2{\d}s} 
\quad\le\quad
C''M^2T\,,
\\
\Expect{\sup_{t \le T \wedge \sigma_M}\left|\int_0^tY_s{\d}s\right|} 
\quad&\le\quad \Expect{\int_0^{T \wedge \sigma_M}|Y_s|{\d}s} 
\quad\le\quad MT\,.
\end{align*}
Under the hypothesis of $(i)$ it follows that, for arbitrary $M \ge 1$ and $T \ge \eps$,
\begin{multline}\label{eq:Bin}
\Prob{\sup_{t \le \tau}\left|\int_0^tY_s{\d}B_s\right| > x} 
 \quad\le \quad
 \Prob{\tau >\sigma_M} + \Prob{\tau >T} +
 \Prob{\sup_{t \le T \wedge \sigma_M}\left|\int_0^tY_s{\d}B_s\right| \ge x}\\
\quad\leq\quad
  CM^{-\alpha} + C \eps^{\beta}T^{-\beta} + C''M^2 T x^{-2}\,,
\end{multline}
where the last inequality follows from the hypothesis of $(i)$ together with a Markov inequality argument.
Similarly,
\begin{multline}\label{eq:Bout}
\Prob{\sup_{t \le \tau}\left|\int_0^tY_s{\d}s\right| \ge x}
\quad \le\quad
\Prob{\tau >\sigma_M} + \Prob{\tau >T} + 
\Prob{\sup_{t \le T \wedge \sigma_M}\left|\int_0^tY_s{\d}s\right| \ge x}\\
\quad \le \quad
CM^{-\alpha} + C \eps^{\beta}T^{-\beta} + MTx^{-1}\,.
\end{multline}
The first assertion of $(i)$ now follows by optimization. To be explicit, set $T=\sqrt{\eps x}$ and $M=\eps^{-1/8}\sqrt{x}$ in \eqref{eq:Bin}
and use $\eps \in (0,1)$ (second inequality) followed by $x \ge \eps^{1/4}$ (third inequality)
to obtain
\begin{multline*}
\Prob{\sup_{t \le \tau}\left|\int_0^tY_s{\d}B_s\right| > x} 
\quad\le\quad
C\left(\frac{\eps^{1/4}}{x}\right)^{\alpha/2} + C\left(\frac{\eps}{x}\right)^{\beta/2} + C''\left(\frac{\eps^{1/2}}{x}\right)^{1/2}
\\
\quad\le\quad
C\left(\frac{\eps^{1/4}}{x}\right)^{\alpha/2} + C\left(\frac{\eps^{1/4}}{x}\right)^{\beta/2} + C''\left(\frac{\eps^{1/4}}{x}\right)^{1/2}
\quad\le\quad
3 \max\{C, C''\}\left(\frac{\eps^{1/4}}{x}\right)^{(\alpha \wedge \beta \wedge 1)/2}\,.
\end{multline*}
The second assertion of $(i)$ follows similarly: set $T=\sqrt{\eps}x^{1/4}$ and $M=\eps^{-1/4}x^{1/4}$ in \eqref{eq:Bout},
and use $\eps \in (0,1)$ (second inequality) followed by $x \ge \eps^{1/2}$ (third inequality)
to obtain
\begin{multline*}
\Prob{\sup_{t \le \tau}\left|\int_0^tY_s{\d}s\right| \ge x}
\quad\le\quad
C\left(\frac{\eps}{x}\right)^{\alpha/4} + C\left(\frac{\eps^2}{x}\right)^{\beta/4} + \left(\frac{\eps^{1/2}}{x}\right)^{1/2}
\\
\quad\le\quad
C\left(\frac{\eps^{1/2}}{x}\right)^{\alpha/4} + C\left(\frac{\eps^{1/2}}{x}\right)^{\beta/4} + \left(\frac{\eps^{1/2}}{x}\right)^{1/2}
\quad\le\quad
3 \max\{C, 1\}\left(\frac{\eps^{1/2}}{x}\right)^{(\alpha \wedge \beta \wedge 2)/4}\,.
\end{multline*}
The proof of $(ii)$ follows along similar lines (using $M = \sqrt{\eps x}$ and $T=\eps^{-1/2}\sqrt{x}$ for the first assertion and $M=\sqrt{\eps}x^{1/4}$ and $T=\eps^{-1/4}x^{1/4}$ for the second assertion).
\end{proof}
\begin{lem}\label{lem:tail2}
Let $X_i, \tau_i$ be non-negative random variables adapted to a given filtration
$(\mathcal{F}_i: {i \ge 1})$ 
and satisfying
\begin{align}\label{eq:Xdec}
\Prob{X_{i+1} > x \mid \mathcal{F}_i} \quad& \le\quad C_{\alpha}x^{-\alpha}\,,
\\
\label{eq:taudec}
\Prob{\tau_{i+1} > t \mid \mathcal{F}_i} \quad& \le\quad C_{\beta} t^{-\beta}\,,
\end{align}
for some $\alpha, \beta >0$ and $x, t \ge 1$, where $C_{\alpha}, C_{\beta}$ are positive constants that do not depend on $i$.
Then for any $\gamma < \alpha 
\wedge \beta$ there is $\eps_0 >0$ (depending on $\alpha$, $\beta$ and $\gamma$) such that
$$
\Prob{\tau_1 + \sum_{k=1}^{\infty} \eps_0^k (\Pi_{j=1}^kX_j)\tau_{k+1} >t}
\quad\le\quad
C'_{\beta} t^{-\gamma}
$$
for some constant $C'_{\beta}$ depending only on $\beta$, and $t \ge 1$.
\end{lem}
\begin{proof}
Set $\Pi_k = \Pi_{j=1}^kX_j$ for $k \ge 1$, with $\Pi_0=1$.
Take any $\gamma < \alpha \wedge \beta$.
Using $\Expect{X_{i+1}^{\gamma} \mid \mathcal{F}_i} \le 1 + \int_1^{\infty}\Prob{X_{i+1}^{\gamma} > x\mid \mathcal{F}_i} dx$, it follows from \eqref{eq:Xdec} that $\Expect{X_{i+1}^{\gamma} \mid \mathcal{F}_i} \le 1 + \frac{\gamma C_{\alpha}}{\alpha-\gamma}< \infty$.
Hence, for $\eps>0$,
\begin{align}\label{eq:moment}
\Expect{(\eps^k\Pi_k)^{\gamma}} 
\quad\le\quad
\left[\eps\left(1 + \frac{\gamma C_{\alpha}}{\alpha-\gamma}\right)^{1/\gamma}\right]^{k\gamma}\,.
\end{align}
For any $\eps>0$, we can write for any $t \ge 1$,
\begin{multline}\label{eq:first}
\Prob{\sum_{k=0}^{\infty}\eps^k\Pi_k\tau_{k+1} > t^{\beta/\gamma}} \quad \le\quad
\\
\Prob{\sum_{k=0}^{\infty}\eps^k\Pi_k\tau_{k+1} > \sum_{k=0}^{\infty}2^{-k}(\eps^k\Pi_k)^{1-\frac{\gamma}{\beta}}t} + \Prob{\sum_{k=0}^{\infty}2^{-k}(\eps^k\Pi_k)^{1-\frac{\gamma}{\beta}} > t^{\frac{\beta}{\gamma}-1}}
\end{multline}
(adopting the convention that $\Prob{Z> \infty}=0$ for any random variable $Z$).
Take $\eps_0 >0$ satisfying $\eps_0 (1 + \frac{\gamma C_{\alpha}}{\alpha-\gamma})^{1/\gamma} = 4^{-\beta/\gamma}$.
Then
\begin{multline*}
\Prob{\sum_{k=0}^{\infty}\eps_0^k\Pi_k\tau_{k+1} > \sum_{k=0}^{\infty}2^{-k}(\eps_0^k\Pi_k)^{1-\frac{\gamma}{\beta}}t} 
\quad\le\quad
\sum_{k=0}^{\infty}\Prob{\tau_{k+1} > 2^{-k} (\eps_0^k\Pi_k)^{-\gamma/\beta}t}\\
\quad \le\quad 
\max\{1,C_{\beta}\}t^{-\beta}\sum_{k=0}^{\infty}2^{k\beta}\Expect{(\eps_0^k\Pi_k)^{\gamma}}
\quad \le\quad
\max\{1,C_{\beta}\} t^{-\beta} \sum_{k=0}^{\infty} 2^{-k\beta} 
\quad=\quad
C'_{\beta}t^{-\beta}\,,
\end{multline*}
where the second inequality is obtained using \eqref{eq:taudec}, and the third by using \eqref{eq:moment} together with the specific choice of $\eps_0$.
We have used $\max\{1,C_{\beta}\}$ in place of $C_{\beta}$ to account for the situation when $2^{-k} (\eps_0^k\Pi_k)^{-\gamma/\beta}t < 1$.
Furthermore
\begin{align*}
\Prob{\sum_{k=1}^{\infty}2^{-k}(\eps_0^k\Pi_k)^{1-\frac{\gamma}{\beta}} > t^{\frac{\beta}{\gamma}-1}}
\quad& \le\quad
\sum_{k=1}^{\infty}\Prob{(\eps_0^k\Pi_k)^{\gamma} > t^{\beta}}\\
\quad& \le\quad
t^{-\beta}\sum_{k=1}^{\infty}\Expect{(\eps_0^k\Pi_k)^{\gamma}}
\quad\le\quad
t^{-\beta}\sum_{k=1}^{\infty}4^{-k\beta} = C'_{\beta}t^{-\beta}\,,
\end{align*}
where we have used the Markov inequality to obtain the second inequality above and \eqref{eq:moment} together with the choice of $\eps_0$ for the third inequality.
 
The above estimates can be used with \eqref{eq:first} to show
$$
\Prob{\sum_{k=1}^{\infty}\eps_0^k\Pi_k\tau_{k+1} > t^{\beta/\gamma}} \quad\le\quad 2C'_{\beta}t^{-\beta}\,,
$$
which proves the lemma.
\end{proof}

\section[Coupling BM(R2) and a single monomial Stratonovich integral]%
{Coupling BM($\Reals^2$) and a single monomial Stratonovich integral}\label{sec:coup}
In this section, we construct couplings of $(W_1, W_2, I_{(a,b)})$ and $(\tW_1, \tW_2, \tI_{(a,b)})$ 
for $(a,b) \in \Delta_n, a \ge 1, b \ge 0$.
The 
cases $a+b=1$ (which implies $a=1$, $b=0$) and $a+b>1$ 
differ in complexity,
so we first consider the simpler case $a+b=1$ (Lemma \ref{lem:Heisenberg}).
This case is significantly easier to describe,
and corresponds to the case of Brownian motion on the Heisenberg group already treated in 
\cite{BenArousCranstonKendall-1995} and \cite{Kendall-2007} as noted in Remark \ref{rem:Heisen} below;
however the present
technique 
carries 
through to the case $a+b>1$ (Lemma \ref{lem:couple1}).
Thus, the construction given in the simplest non-trivial case (Lemma \ref{lem:Heisenberg})
is a good model for the general approach. 
Lemma \ref{lem:couple1} deals with coupling just one monomial Stratonovich integral of more general form,
but this is an essential component of the inductive argument that will be required
to establish coupling for a finite set of monomial Stratonovich integrals in Section \ref{sec:coup-multiple}.

We will use some further notation, namely
$\Delta W_1=W_1-\tW_1$ and $\Delta I_{(a,b)}=I_{(a,b)}-\tI_{(a,b)}$.

\subsection{Case of simplest non-trivial monomial Stratonovich integral}\label{sec:coup-simple}
The next lemma establishes a coupling result
based on a driving \(2\)-dimensional Brownian motion \(W_1, W_2\) plus the single monomial stochastic integral 
\(I_{(1,0)}\).
\begin{lem}\label{lem:Heisenberg}
For any $\gamma < \frac{1}{3}$, there exists a successful Markovian coupling 
$\mathbb{P}_{\gamma}$ of $(W_1, W_2, I_{(1,0)})$ 
and $(\tW_1, \tW_2, \tI_{(1,0)})$ started from distinct points $(w_1, w_2, \fki)$ and $(w_1, w_2, \tfki)$ respectively, 
with coupling time $\CouplingTime_{\gamma}$, satisfying
\begin{align}\label{eq:Htail}
\sup_{w_1, w_2, | \fki-\tfki | \le 1}\mathbb{P}_{\gamma}\left[\CouplingTime_{\gamma} > t\right] \le C_{\gamma}t^{-\gamma}, \ \ t\ge1.
\end{align}
\end{lem}
\begin{proof}
We first outline the general proof strategy.
At all times
$W_2$ and $\tW_2$ 
will be synchronously coupled; hence
$W_2(t)=\tW_2(t)$ for all $t\ge 0$.
Brownian scaling as given in Lemma \ref{lem:scaling}
can be used to re-scale to a unit difference between the two stochastic integrals,
thus reducing all cases to the case of
starting points $(w_1,w_2,\fki)$ and $(w_1,w_2,\fki-1)$ for $w, i \in \Reals$.
The coupling decomposes naturally into disjoint {cycles}.
Each cycle consists of a patterned alternation between phases of reflection and synchronous coupling 
for $W_1$ and $\widetilde{W}_1$,
so that the distance between the coupled processes $(W_1, W_2, I_{(1,0)})$ and $(\tW_1, \tW_2, \tI_{(1,0)})$ at the end of the cycle 
is roughly a fixed proportion of the distance between them at the start of the cycle.
At the end of each cycle, the next cycle is constructed by applying the same coupling strategy as the previous cycle 
after appropriately re-scaling the coupled processes \emph{via} Lemma \ref{lem:scaling},
so that there is unit re-scaled distance between them at the start of the next cycle.
Lemma \ref{lem:tail2} is then used to show that the end-points of these cycles have an accumulation point which corresponds to a finite coupling time.
As the coupling strategy within each cycle is the same (modulo re-scaling), 
it is sufficient to describe in detail only the construction of the first cycle. 
Note that iterated cycles and re-scaling to achieve successful coupling 
have been used
to couple Kolmogorov diffusions by \citet{BenArousCranstonKendall-1995}, \citet{KendallPrice-2004}, \citet{BanerjeeKendall-2015}.


\medskip\noindent
\textbf{A: {Description of the first cycle}}\\
As noted before, the scaling argument represented by Lemma \ref{lem:scaling}
shows there is no loss of generality in assuming that $|\Delta I_{(1,0)}(0)| = 1$.
Choose and fix a constant $R>1$. The estimates derived for the first cycle will be uniform with respect to $R>1$ and an optimal choice of $R$ will be made at the end of the proof.
In the proof, $C, C_1, C_2, \dots$ will denote generic positive constants whose values will not depend on $R, w_1, w_2, \fki$ 
and whose value might change from line to line.
The first cycle consists of three phases whose end-points are defined by the following stopping times:
\begin{tabbing}
1:  \quad\= \(T_1 \quad=\quad \inf\{t \ge 0: |\Delta W_1(t)| = R^{-1}\}\),\qquad\= reflection till \(W_1-\widetilde{W}_1\) hits \(\pm R^{-1}\);\\
2:       \> \(T_2 \quad=\quad \inf\{t \ge T_1: \Delta I_{(1,0)}(t)=0\}\),          \> synchronous till \(I_{(1,0)}-\widetilde{I}_{(1,0)}\) hits \(0\);\\
3:       \> \(T_3 \quad=\quad \inf\{ t \ge T_2: \Delta W_1(t) = 0\}\),        \> reflection till \(W_1-\widetilde{W}_1\) hits \(0\).
\end{tabbing}

\medskip\noindent
\textbf{Phase 1:} 
Using Brownian scaling, 
independent
Brownian
increments,
and eigenvalues of the 
Laplacian with Dirichlet
boundary conditions 
on \([-1,1]\), 
together with \(W_1(0)=\widetilde{W}_1(0)\),
it follows that
\begin{align}\label{eq:Htime1}
\Prob{T_1>t/R^2} \quad\le\quad C e^{-\pi^2 t /8}\,.
\end{align}
Now consider the increment of $\Delta I_{(a,b)}$ over the time interval $[0, T_1]$.
Since the second Brownian coordinates satisfy \(W_2=\widetilde{W}_2\) throughout the entire coupling,
and the first Brownian coordinates \(W_1\), \(\widetilde{W}_1\) are reflection coupled hence independent of \(W_2=\widetilde{W}_2\), 
we may re-write the Stratonovich integral
for the increment as an It\^o integral:
$$
\Delta I_{(1,0)}(T_1) - \Delta I_{(1,0)}(0) \quad=\quad \int_{0}^{T_1} \Delta W_1(s){\d}W_2(s)
\,.
$$
On the other hand, $\sup_{t \in [0, T_1]}|\Delta W_1(t)| =\frac{1}{R}$ by definition of \(T_1\).
Using the \(L^2\)-isometry of the It\^o integral,
$$
\Expect{\left|\int_{0}^{T_1} \Delta W_1(s){\d}W_2(s)\right|^2} 
\quad=\quad  
\Expect{\int_{0}^{T_1} (\Delta W_1(s))^2{\d}s} 
\quad\le\quad 
\frac{1}{R^2}\Expect{T_1} 
\quad\le\quad \frac{C}{R^4}\,,
$$
where the last inequality follows from \eqref{eq:Htime1} using \(\Expect{T_1}=\int_0^\infty\Prob{T_1>t}{\d}t\).

By a Markov inequality argument, it now follows
for any $x >0$ that
\begin{align*}
\Prob{|\Delta I_{(1,0)}(T_1) - \Delta I_{(1,0)}(0)| > x} \quad\le\quad \frac{C}{R^4x^2}\,.
\end{align*}
But we have assumed that $|\Delta I_{(1,0)}(0)| = 1$, so for any $x \ge 2$
\begin{align}\label{eq:H1}
\Prob{|\Delta I_{(1,0)}(T_1)| > x} \quad\le\quad 
\Prob{|\Delta I_{(1,0)}(T_1) - \Delta I_{(1,0)}(0)| > x-1} \quad&\le\quad 
\frac{C}{R^4(x-1)^2}\nonumber\\
 &\;\le\;
\frac{4C}{R^4x^2}\,.
\end{align}

\medskip\noindent
\textbf{Phase 2: }
Because we are constructing a Markovian coupling, we may
condition on the past of the driving Brownian motions till time $T_1$.
Under synchronous coupling 
the Stratonovich expression for the increment of $\Delta I_{(a,b)}$ over the time interval $[T_1, T_2]$
can again be re-written as an It\^o integral, only now \(\Delta W_1(s)=\Delta W_1(T_1)\) while \(s\in[T_1, T_2]\).
Thus in this time interval \(\Delta I_{(a,b)}(s)-\Delta I_{(a,b)}(T_1)=\Delta W_1(T_1) \times (W_2(s)-W_2(T_1))\).
Since \(\Delta W_1(T_1)=R^{-1}\) by construction of \(T_1\),
it follows that
$T_2-T_1$ has the same distribution as 
the hitting time of a one-dimensional Brownian motion on the level $-R\sgn(\Delta W_1(T_1)) \Delta I_{(1,0)}(T_1)$.
Thus, for \(x \ge 2\) and \(t>0\), we can assert that
\begin{multline*}
\Prob{T_2-T_1 >t}
\quad\le\quad \Prob{|\Delta I_{(1,0)}(T_1)| > x} + \Prob{T_2-T_1 >t \text{ and } |\Delta I_{(1,0)}(T_1)| \le x}\\
\quad\le\quad \frac{4C}{R^4x^2} + \frac{CRx}{\sqrt{t}}\,,
\end{multline*}
where the last inequality is a consequence of \eqref{eq:H1} 
and a hitting time estimate for Brownian motions derived from the reflection principle.
Taking $x=t^{1/6}$ in the above expression and recalling that $R>1$,
\begin{equation*}
\Prob{T_2-T_1 >t} \quad\le\quad \frac{4C}{R^4t^{1/3}} + \frac{CR}{t^{1/3}} \quad \le \quad \frac{5CR}{t^{1/3}}
\qquad \text{for }t \ge 2^6\,.
\end{equation*}
The above expression gives a useful bound on the probability $\Prob{T_2-T_1 >t}$ only when $t \ge (5CR)^3$. 
We therefore adjust the above bound (using $C$ as a new generic positive constant):
\begin{equation}\label{eq:Htime2}
\Prob{T_2-T_1 >t} \quad\le\quad \frac{CR}{t^{1/3}}
\qquad \text{for }t \ge CR^3\,.
\end{equation}
Note that $\Delta I_{(1,0)}(T_2)=0$ follows from the definition of $T_2$.

\medskip\noindent
\textbf{Phase 3: }
Using reflection coupling, and conditioning on the past at time \(T_2\),
we may view $T_3-T_2$ as the hitting time of level $\frac{1}{2R}$ by a standard Brownian motion.
Employing the reflection principle for Brownian motion
\begin{equation}\label{eq:Htime3}
\Prob{T_3-T_2 >t/R^2} \quad\le\quad 
\sqrt{\frac{2}{\pi}}\frac{1}{\sqrt{t}} \qquad \text{for } t \ge 1\,.
\end{equation}
Moreover, for $x > R^{-2}, H> R^{-2}$, 
and once again re-writing the Stratonovich integral of \(\Delta I_{(1,0)}\) as an It\^o integral,
\begin{align*}
\Prob{|\Delta I_{(1,0)}(T_3)| > x} \quad&\le\quad \Prob{T_3-T_2 >H} + \Prob{|\Delta I_{(1,0)}(T_3)| > x, T_3-T_2 \le H}
\\
\quad&\le\quad \Prob{T_3-T_2 >H} + \Prob{\sup_{T_2 \le t \le T_2 +H} \left|\int_{T_2}^t (\Delta W_1)(s){\d}W_2(s)\right| > x}
\\
\quad&\le\quad \sqrt{\frac{2}{\pi}}\frac{1}{R\sqrt{H}} +\frac{CH}{x^2} \Expect{\sup_{T_2 \le t \le T_2 +H}|\Delta W_1 (t)|^2}\\ 
&\qquad\qquad\qquad\quad\text{ (by } \eqref{eq:Htime3}, \text{ Tchebychev and BDG inequalities)}
\\
\quad&\le\quad \sqrt{\frac{2}{\pi}}\frac{1}{R\sqrt{H}} + \frac{CH}{x^2}\left(\frac{1}{R^2} + H\right)\\ 
&\quad\text{ (using } |\Delta W_1(T_2)|=R^{-1} \text{ and Doob's } L^2\text{-maximal inequality)}
\\
\quad&\le\quad  \sqrt{\frac{2}{\pi}}\frac{1}{R\sqrt{H}} + \frac{CH^2}{x^2} \ \ \left(\text{as } H> \frac{1}{R^2}\right).
\end{align*}
Taking $H=\frac{x^{4/5}}{R^{2/5}}$, we obtain a bound
\begin{equation}\label{eq:H2}
\Prob{|\Delta I_{(1,0)}(T_3)| > x} 
\quad\le\quad
\frac{C}{(R^2x)^{2/5}}\,, \qquad\text{ for } x > \frac{1}{R^2}\,.
\end{equation}
Moreover, the combined effect of the estimates in \eqref{eq:Htime1}, \eqref{eq:Htime2} and \eqref{eq:Htime3}
can be summarized as
\begin{equation}\label{eq:Htime}
\Prob{T_3 > t} \quad\le\quad \frac{CR}{t^{1/3}}, \qquad\text{ for } t \ge CR^3\,.
\end{equation}
The estimates \eqref{eq:H2} and \eqref{eq:Htime} give bounds on 
the difference of the integrals $I_{(1,0)}$ and $\tI_{(1,0)}$ at the end of the first cycle and 
the time taken to complete the first cycle respectively.

\medskip\noindent
\textbf{B: Describing subsequent cycles and successful coupling}\\
%
For $t \ge T_3$, define further stopping times \(T_{k}, k > 3\), such that for any $k \ge 1$, $$|\Delta I_{(1,0)}(T_{3k})|^{-1}T_{3k+j}, \ j=1,2,3,$$ is the time of completion of the \(j^\text{th}\) phase
of the first cycle constructed above for the re-scaled processes $$(|\Delta I_{(1,0)}(T_{3k})|^{-1/2}W_1(T_{3k}+t), |\Delta I_{(1,0)}(T_{3k})|^{-1/2}W_2(T_{3k}+t), |\Delta I_{(1,0)}(T_{3k})|^{-1}I_{(0,1)}(T_{3k}+t))_{t \ge 0}$$ and $$(|\Delta I_{(1,0)}(T_{3k})|^{-1/2}\widetilde{W}_1(T_{3k}+t), |\Delta I_{(1,0)}(T_{3k})|^{-1/2}\widetilde{W}_2(T_{3k}+t), |\Delta I_{(1,0)}(T_{3k})|^{-1}\widetilde{I}_{(0,1)}(T_{3k}+t))_{t \ge 0}$$ in place of $(W_1(t), W_2(t), I_{(0,1)}(t))_{t \ge 0}$ and $(\widetilde{W}_1(t), \widetilde{W}_2(t), \widetilde{I}_{(0,1)}(t))_{t \ge 0}$ respectively.

The concatenation of these cycles does in fact lead to a successful coupling.
The proof of this follows from two facts:
\begin{itemize}
\item[(i)] $\lim_{k \rightarrow \infty}\Delta I_{(1,0)}(T_{3k}) = 0$, 
meaning that the coupled processes, observed at the end-points of the cycles, 
come arbitrarily close as the number of cycles becomes large, and 
\item[(ii)] $\lim_{k \rightarrow \infty}T_{3k} < \infty$ almost surely, 
meaning that the end points of these cycles have a finite accumulation point \(T_\infty\), 
so that the concatenation completes in finite time.
\end{itemize}
Continuity of Brownian motion and stochastic integrals then implies successful coupling at time \(T_\infty\).

We now demonstrate that these two facts follow from Lemma \ref{lem:tail2}.
Define
\begin{equation}
\tau^*_k = \frac{T_{3k} - T_{3k-3}}{|\Delta I_{(1,0)}(T_{3k-3})|},\quad
X^*_k = \frac{R^2|\Delta I_{(1,0)}(T_{3k})|}{|\Delta I_{(1,0)}(T_{3k-3})|}, \quad \text{ for } k \ge 1\,,
\label{eq:perpetuity-definitions}
\end{equation}
where we take $T_0=0$.
Applying
Lemma \ref{lem:scaling}
to \eqref{eq:H2}, 
the pair $(X^*_k, 1)$ satisfies the hypotheses of $(X_k, \tau_k)$ of Lemma \ref{lem:tail2} with $\alpha=2/5$ and any $\beta > 0$.
Thus, Lemma \ref{lem:tail2} implies that there is $R_0 > 1$ such that for all $R \ge R_0$,
$$
\sum_{k=1}^{\infty}R^{-2k}\left(\Pi_{j=1}^kX^*_j\right) \quad<\quad \infty, \ \text{almost surely}.
$$
Choosing $R \ge R_0$ in the description of the first cycle, $|\Delta I_{(1,0)}(T_{3k})| = R^{-2k}\left(\Pi_{j=1}^kX^*_j\right)$ 
and consequently, $\lim_{k \rightarrow \infty}\Delta I_{(1,0)}(T_{3k}) = 0$ almost surely.

To prove that the coupling is successful in finite time almost surely
and that the coupling time has a power law tail given by \eqref{eq:Htail}, 
apply
Lemma \ref{lem:scaling}
to \eqref{eq:H2} and \eqref{eq:Htime}:
$(X^*_k, \tau^*_k/R^3)$ satisfies the hypotheses of $(X_k, \tau_k)$ of Lemma \ref{lem:tail2} with $\alpha=2/5$ 
and $\beta = 1/3$.
Thus, by Lemma \ref{lem:tail2}, for any $0<\gamma < 1/3$, there is $R_{\gamma} > 1$ such that for any $R \ge R_{\gamma}$,
\begin{align}\label{eq:Heisenbergtime}
\Prob{\tau^*_1 + \sum_{k=1}^{\infty}R^{-2k}\left(\Pi_{j=1}^kX^*_j\right)\tau^*_{k+1} > R^3 t} \quad\le\quad C t^{-\gamma}\,.
\end{align}
Now observe the following product collapses because of the definitions expressed by \eqref{eq:perpetuity-definitions}:
\[
 R^{-2k} \left(\Pi_{j=1}^kX^*_j\right)\tau^*_{k+1}
 \quad=\quad
 (T_{3(k+1)}-T_{3k}) / \Delta I_{(0,1)}(T_0)\,.
\]
Thus, for any $0<\gamma< 1/3$,
the above coupling construction with $R = \operatorname{max}\{R_0,R_{\gamma}\}$
gives the required successful coupling satisfying \eqref{eq:Htail}.
\end{proof}

\begin{rem}\label{rem:Heisen}
Recall the Brownian motion in the Heisenberg group started at $(w_1,w_2, \fki)$, 
defined as the ${\Reals}^3$ valued process given by $$((W_1(t), W_2(t), \fki + \int_0^tW_1(s){\d}W_2(s) - \int_0^tW_2(s){\d}W_1(s)): t \ge 0),$$ 
where $(W_1, W_2)$ is a two-dimensional Brownian motion started at $(w_1,w_2)$.
Lemma \ref{lem:Heisenberg} is of independent interest as it gives a successful Markovian coupling of Brownian motions on the Heisenberg group 
started at $(w_1,w_2, \fki)$ and $(w_1,w_2, \tfki)$ with explicit bounds on the tail probabilities of the coupling time.
To see this, note that by the It\^o formula,
$$
W_1(t)W_2(t)-W_1(0)W_2(0)\quad = \quad \int_0^tW_1(s){\d}W_2(s) + \int_0^tW_2(s){\d}W_1(s)\,, \qquad \text{ for } t \ge 0.
$$
From this, we obtain
\begin{align*}
 \fki + \int_0^tW_1(s){\d}W_2(s) & - \int_0^tW_2(s){\d}W_1(s)\\
 \quad&=\quad 2\left(\frac{\fki}{2}+ \int_0^tW_1(s){\d}W_2(s)\right)\\
 & \quad \qquad - \left(\int_0^tW_1(s){\d}W_2(s) + \int_0^tW_2(s){\d}W_1(s)\right)\\
 \quad&=\quad 2\left(\frac{\fki}{2}+ \int_0^tW_1(s){\d}W_2(s)\right) - \left(W_1(t)W_2(t)-W_1(0)W_2(0)\right)\,.
\end{align*}
Thus, the successful coupling construction given in Lemma \ref{lem:Heisenberg} for
$(W_1, W_2, I_{(1,0)})$ and $(\tW_1, \tW_2, \tI_{(1,0)})$, started from $(w_1,w_2, \fki/2)$ and $(w_1,w_2, \tfki/2)$ respectively, 
is also a successful coupling of the corresponding Brownian motions on the Heisenberg group started from $(w_1,w_2, \fki)$ and $(w_1,w_2, \tfki)$.
Couplings of Brownian motions on the Heisenberg group have appeared in several papers in recent times: 
\cite{BenArousCranstonKendall-1995} and \cite{Kendall-2007} have constructed
successful Markovian couplings of Brownian motions for the Heisenberg group.
\citet[Theorem 3.1]{Kendall-2009d} established some coupling time distribution asymptotics for the coupling constructed in \cite{Kendall-2007}, 
under some limiting operation on the starting points.
But our result gives explicit bounds on the tail probabilities of the coupling time for each $t$ 
and each pair of starting points $(w_1,w_2, \fki)$ and $(w_1,w_2, \tfki)$ 
(in fact, this coupling can be extended to general pairs of distinct starting points $(w_1,w_2, \fki)$ 
and $(\tw_1,\tw_2, \tfki)$ and associated bounds can be derived).
Moreover,
the tail probabilities of the coupling time 
of 
\citet[Theorem 3.1]{Kendall-2009d}
decay at best like $t^{-1/6}$;
the rate in Lemma \ref{lem:Heisenberg} is significantly better.
We note here however 
(a) that the treatment by \cite{BenArousCranstonKendall-1995} and \cite{Kendall-2007} uses 
an invariant difference that 
permits a generalization which
couples all possible
stochastic areas for a \(d\)-dimensional Brownian motion \citep{Kendall-2007};
(b) that 
\citet[Lemma 3.1]{BanerjeeGordinaMariano-2016}
obtained a non-Markovian coupling for Brownian motions on the Heisenberg group 
started at $(w_1,w_2, \fki)$ and $(w_1,w_2, \tfki)$ that attains the total variation bound 
(the best possible bound on the tails of coupling time distribution), and 
decays like $t^{-1}$, 
which is significantly better than the bound in Lemma \ref{lem:Heisenberg}.
Moreover Markovian couplings cannot reach a bound that decays faster than $t^{-1/2}$ \citep[Remark 3.2]{BanerjeeGordinaMariano-2016}.
It would be interesting to investigate whether the bound $t^{-1/2}$ can be attained, or whether $t^{-1/3}$ is the best bound for Markovian couplings.
\end{rem}

\subsection{Case of general monomial Stratonovich integral}\label{sec:coup-general}

The next lemma generalizes the previous coupling construction, establishing a coupling result
based on a driving \(2\)-dimensional Brownian motion plus a single monomial stochastic integral: $(W_1, W_2, I_{(a,b)})$ 
for a single fixed \((a,b) \in \Delta_n\) with \(a\ge 1\), \(b \ge 0\) and \(a+b>1\).
Recall from Section \ref{sec:TP} that $f(k,l)=2nk+ (2n+1) l$.

\begin{lem}\label{lem:couple1}
For any $(a,b) \in \Delta_n$ with $a \ge 1$, $b \ge 0$, $a+b>1$, 
there exists $R_0>1$ such that for each $R \ge R_0$, 
we can obtain a successful Markovian coupling construction $\mathbb{P}_R$ 
of $(W_1, W_2, I_{(a,b)})$ and $(\tW_1, \tW_2, \tI_{(a,b)})$ starting from $(w,R w, \fki)$ and $(w, R w, \tfki)$ 
respectively, with coupling time $T_{R,(a,b)}$, such that:
\begin{itemize}
\item[(i)] There are positive constants \(\gamma\), \(C\) not depending on $R$ such that, for large $t$,
\begin{equation}\label{eq:couple1-1}
\sup
\left\{\ProbR{T_{R,(a,b)} >R^{4n+2}t}
\;:\;
w, \fki, \tfki \in \Reals, |\fki-\tfki| \le 1
\right\} 
\quad\le\quad Ct^{-\gamma}\,.
\end{equation}

In the interval $[0,T_{R,(a,b)}]$ we identify the \emph{active region} $S_{R,(a,b)}$,
\begin{equation*}
S_{R,(a,b)}\quad=\quad \text{closure of }\{t \le T_{R,(a,b)}: W_1(t) \neq \tW_1(t)\}.
\end{equation*} 
$S_{R,(a,b)}$ depends on $R, a,b$; both $[0, T_{R,(a,b)}) \cap S_{R,(a,b)}$ 
and $[0, T_{R,(a,b)}) \setminus (\text{interior of } S_{R,(a,b)})$ are
unions of countable sequences of disjoint random closed intervals, where for 
each sequence of intervals the left-end-points of
the intervals form an increasing sequence.
Writing the total length of $S_{R,(a,b)}$ by $|S_{R,(a,b)}|$, the following holds for large $t$,
\begin{equation}\label{eq:couple1-2}
\sup\left\{
\ProbR{|S_{R,(a,b)}| > t}
\;:\; w, \fki, \tfki \in \Reals, |\fki-\tfki| \le 1
\right\}
\quad\le\quad Ct^{-\gamma} 
\,.
\end{equation}
\item[(ii)] There are positive constants \(\alpha\), \(C\) not depending on $R$ such that for large $t$,
\begin{multline}\label{eq:couple1-3}
\sup\left\{
\ProbR{\sup_{t \le T_{R, (a,b)}}|W_1(t)-\tW_1(t)|>x/R^{f(a-1,b)}}
\;:\; w, \fki, \tfki \in \Reals, |\fki-\tfki| \le 1
\right\}\\
\quad\le\quad Cx^{-\alpha} 
\,.
\end{multline}
\end{itemize}
\end{lem}
\noindent For convenience, we will prove the inequalities \eqref{eq:couple1-1}, \eqref{eq:couple1-2} and \eqref{eq:couple1-3} for $t \ge 1$.
\begin{proof}
As before, $W_2$ and $\tW_2$ will be synchronously coupled at all times so we may take $W_2=\tW_2$.
Brownian scaling (Lemma \ref{lem:scaling}) can be applied to ensure the monomial stochastic integrals
differ by 1: so it suffices to consider starting points $(w,Rw,\fki)$ and $(w,Rw,\fki-1)$ for $w, \fki \in \Reals$.
Let $\gamma, \delta, C, C_1, C_2 \dots$ be generic positive constants not depending on $R, w, \fki$,
(but often depending on \(a\) and \(b\))
whose values might change from line to line.


The proof uses some martingale estimates, so we will use the decomposition of the Stratonovich integral
$I_{(a,b)}$ in It\^o integral form:
\begin{align}\label{eq:itostrat}
I_{(a,b)}(t)\quad=\quad I_{(a,b)}(0) +
\int_0^t{W_1(s)^a}{W_2(s)^b}{\d}W_2(s) + \frac{b}{2}\int_0^t{W_1(s)^a}{W_2(s)^{b-1}}{\d}s\,.
\end{align}
In contrast with the case of Lemma \ref{lem:Heisenberg},
here the Stratonovich integral has a drift component if \(b\geq1\).

As in the previous lemma,
the coupling decomposes into disjoint {cycles}, and the successive cycles are connected via scaling.
We describe the first cycle and then
discusses the total effect of this and subsequent cycles on finiteness and
moment estimates for the coupling time.

\medskip\noindent
\textbf{A: Description of the first cycle}\\
The first cycle consists of five phases.
The coupling strategy alternates between synchronous coupling and reflection coupling of $W_1$ and $\tW_1$ between the phases.
We will first describe each phase in terms of an arbitrary value of the tuning parameter $R \ge R_0>1$.
The estimates derived for the first cycle will hold uniformly with respect to $R>1$ and the appropriate lower bound $R_0$ for $R$ will arise in the course of the proof and be specified at the end of the coupling construction.
The end-points of the five phases are defined by the following stopping times.
Initially \(W_1(0)=\tW_1(0)\) (\(W_2=\tW_2\) throughout.)
\begin{tabbing}
1: \quad\=  \(\theta_1=\inf\{t \ge 0: W_2(t)=R W_1(t) \text{ and } |W_1(t)| \ge R^{2n}\}\),\quad\= synchronous till \(W_1\)\\
        \>\> hits \(R^{-1} W_2\) \\
        \>  \>\textbf{and} \(|W_1| \ge R^{2n}\), and\\
        \>  \>note \(W_1(\theta_1)=\tW_1(\theta_1)\);\\
2:      \>  \(\tau_1=\inf\left\{ t \ge \theta_1: |\Delta W_1(t)| =\frac{1}{|W_1(\theta_1)|^{a+b-1}R^b}\right\}\),\> reflection till \(\Delta W_1\) hits\\
        \>\> \(\pm\frac{1}{|W_1(\theta_1)|^{a+b-1}R^b}\),\\
        \>  \>and note \\
        \>  \>\(W_1(\tau_1)-W_1(\theta_1)\)\\
        \>\> \(=-(\tW_1(\tau_1)-\tW_1(\theta_1))\);\\
3:      \>  \(\eta_1=\inf\{t \ge \tau_1: W_2(t)=W_2(\tau_1)-a^{-1}\sgn(\Delta W_1(\tau_1))(\sgn(W_1(\theta_1)))^{a+b-1} \}\),\\
        \>  \> synchronous till \(W_2-W_2(\tau_1)\)\\
        \>\> hits\\
        \>\>  \(-\frac{\sgn(\Delta W_1(\tau_1))(\sgn(W_1(\theta_1)))^{a+b-1}}{a}\),\\
        \>  \>and note\\
        \>  \>\(W_1(\eta_1)-W_1(\tau_1)\)\\
        \>\> \(=\tW_1(\eta_1)-\tW_1(\tau_1)\);\\
4:      \>  \(\lambda_1=\inf\{t \ge \eta_1: \Delta W_1(t) = 0\}\),\> reflection till \(\Delta W_1\) hits \(0\),\\
        \>  \>and note \(W_1(\lambda_1)=\tW_1(\lambda_1)\);\\
5:      \>  \(\beta_1=\inf\{t \ge \lambda_1: W_2(t)=R W_1(t)\}\),\> synchronous till \(R W_1-W_2\)\\
         \>\> hits \(0\),\\
        \>  \>and note \(W_1(\beta_1)=\tW_1(\beta_1)\).
\end{tabbing}
Note that the first and last phases both use synchronous coupling. 
However we do not amalgamate these across cycles,
since at the \(\beta_k\) times we have \(R W_1=W_2\) as well as \(W=\tW\).


\medskip\noindent
\textbf{Phase 1: }
In this phase, synchronous coupling of $W_1$ and $\tW_1$ is applied on $[0, \theta_1]$, where $\theta_1$ is the stopping time defined above.
The Brownian motions agree at time \(0\) and 
are synchronously coupled on $[0, \theta_1]$, so agree over the whole interval $[0, \theta_1]$.
Therefore $\Delta I_{(a,b)}(\theta_1)=\Delta I_{(a,b)}(0)=1$.

We first estimate the tail probability of $\theta_1$ as follows. If \(t\geq1\) then
\begin{align}\label{eq:thetatail}
\Prob{\theta_1> R^{4n+2} t} \quad& \le\quad \frac{C\log t}{\sqrt{t}}\,.
\end{align}
To see this, note that \(\theta_1\) is obtained by starting a planar Brownian motion located at distance \(\sqrt{1+R^2}w\) along
the line \(W_2=RW_1\) from the origin,
and running it till it hits the diagonal \(W_2=R W_1\) at a distance at least \(R^{2n}\) from the origin.
By Brownian scaling and rotational invariance of planar Brownian motion, 
$\theta_1$ is stochastically dominated by $R^{4n+2} \theta'_1$, where
$$
\theta'_1\quad=\quad \inf\{t>0: W^*_1(t)=0, |W^*_2(t)| \ge \sqrt{2}\}\,.
$$
Here $(W^*_1, W^*_2)$ is
a planar Brownian motion $(W^*_1, W^*_2)$ with $W^*_1(0)=0, W^*_2(0)=\frac{\sqrt{1 + R^2}w}{R^{2n+1}}$;
if  \(\sqrt{2}\) were replaced by \(\sqrt{1+R^{-2}}\) then the stochastic domination would become an equality
(recall, \(R>1\)).
If $L(t)$ denotes the local time of $W^*_1$ at $0$ at time $t$ and $\zeta(t)$ denotes the inverse local time, 
then $W^*_2(\zeta(t))=\mathcal{C}(t)$, where $\mathcal{C}$ is a Cauchy process starting at $\sqrt{1 + R^2}w/R^{2n+1}$.
If $L(\theta'_1)>s$, then the continuity of $L$ implies $\theta'_1>\zeta(s)$.
The range of $\zeta$ is a subset of the set of times where the monotone function $L$ increases
(namely, the times where $W^*_1=0$), so $\theta'_1>\zeta(s)$ yields $\sqrt{2} > W^*_2(\zeta(s))$ from the definition of $\theta'_1$.
Hence, for $t \ge 1$, and $u=C_2^{-1}\log t$ for a certain positive constant \(C_2\),
\begin{multline}\label{eq:the}
\Prob{\theta'_1>t} \quad\le\quad
\Prob{L(t)\le u} + \Prob{\theta'_1>t, L(t) > u}\\
\quad\le\quad \Prob{L(t)\le u} + \Prob{L(\theta'_1) > u}
\quad\le\quad \Prob{L(t)\le u} + \Prob{\sup_{s \le u} |\mathcal{C}(s)| \le \sqrt{2}}\,.
\end{multline}
By the L\'evy transform, the local time process $(L(s): s \ge 0)$ has the distribution of the running supremum of Brownian motion, so
$$
\Prob{L(t)\le u} \quad\le\quad \frac{2}{\pi} \frac{u}{\sqrt{t}}\,.
$$
To bound the second probability in \eqref{eq:the}, recall that the Cauchy process $\mathcal{C}$ is a pure jump L\'evy process.
Consequently, the increments $(\mathcal{C}(j)-\mathcal{C}(j-1) : 1 \le j \le \lfloor u \rfloor)$ are {i.i.d.} 
with a common Cauchy distribution.
If $\sup_{s \le u} |\mathcal{C}(s)| \le \sqrt{2}$ holds then 
$|\mathcal{C}(j)-\mathcal{C}(j-1)| \le 2\sqrt{2}$ for $1 \le j \le \lfloor u \rfloor$, 
and therefore  (using positive constants $C_1, C_2$ not depending on $w$)
$$
\Prob{\sup_{s \le u} |\mathcal{C}(s)| \le \sqrt{2}} \quad\le\quad C_1 e^{-C_2 u}\,.
$$
Applying these bounds to \eqref{eq:the},
$$
\Prob{\theta'_1>t} \quad\le\quad \frac{\sqrt{2}u}{\sqrt{\pi t}} + C_1 e^{-C_2 u}\,.
$$
The required bound \eqref{eq:thetatail} follows by taking $u=C^{-1}_2\log t$
and choosing a suitable
\(C\) bearing in mind that \(t\geq1\).

\medskip\noindent
\textbf{Phase 2: }
This phase employs reflection coupling between $W_1$ and $\tW_1$,
and runs from time $\theta_1$ till the stopping time
$$
\tau_1 \quad=\quad
\inf\left\{ t \ge \theta_1: |\Delta W_1(t)| = \frac{1}{|W_1(\theta_1)|^{a+b-1}R^b}\right\}\,.
$$
Phase 1 leaves $|W_1(\theta_1)| \ge R^{2n}$. Using $f(a-1, b) = 2n(a-1) + (2n+1)b$,
$$
\sup_{t \in [\theta_1, \tau_1]}|\Delta W_1(t)| 
\quad=\quad  \frac{1}{|W_1(\theta_1)|^{a+b-1}R^b} \le \frac{1}{R^{2n(a+b-1)}R^b}=\frac{1}{R^{f(a-1,b)}}\,.
$$
Thus, for $t \ge 1$, applying successively reflection coupling and Brownian scaling,
\begin{align}\label{eq:tautail}
\Prob{\tau_1-\theta_1>t/R^{2f(a-1,b)}} 
\quad&\le\quad \Prob{\sup_{0\leq s-\theta_1\leq t/R^{2f(a-1,b)}}|\Delta W_1(s)| \;\le\; \frac{1}{R^{f(a-1,b)}}}\nonumber\\
\quad&=\quad \Prob{\sup_{0\leq s-\theta_1\leq t/R^{2f(a-1,b)}}|W_1(s)-W_1(\theta_1)| \;\le\; \frac{1}{2R^{f(a-1,b)}}}\nonumber\\
\quad&=\quad \Prob{\sup_{s \in [0,t]}|W_1(s)| \;\le\; \frac{1}{2}}\quad\le\quad C_1 e^{-C_2t}\,.
\end{align}

Consider the telescoping sum (for $s \ge 0$),
\begin{align}\label{eq:expan1}
({W_1(s)^a} & - {\tW_1(s)^a}){W_2(s)^b}\nonumber\\
\quad&=\quad \Delta W_1(s)({W_1(s)^{a-1}}+ {W_1(s)^{a-2}}\tW_1(s) + \dots + {\tW_1(s)^{a-1}}){W_2(s)^b}.
\end{align}
Since $W_2(\theta_1)=R W_1(\theta_1)$, we know 
$$\sup_{s \in [\theta_1, \tau_1]}|\Delta W_1(s)| =  \frac{1}{|W_1(\theta_1)|^{a+b-1}R^b} =  \frac{1}{|W_1(\theta_1)|^{a-1}|W_2(\theta_1)|^{b}}.$$
So, for $1 \le k \le a$, 
\begin{multline}\label{eq:bublu}
\Prob{\sup_{\theta_1 \le s \le \tau_1}|\Delta W_1(s)||W_1(s)|^{a-k}|\tW_1(s)|^{k-1}|W_2(s)|^b > x} 
\quad\le\quad
\\
\Prob{\sup_{\theta_1 \le s \le \tau_1}
\left|\frac{W_1(s)}{W_1(\theta_1)}\right|^{a-k}
    \left|\frac{\tW_2(s)}{\tW_1(\theta_1)}\right|^{k-1}
    \left|\frac{W_2(s)}{ W_2(\theta_1)}\right|^b > x}\,.
\end{multline}
As $W_1$ and $\tW_1$ are reflection coupled in $[\theta_1, \tau_1]$, therefore
\begin{equation}\label{eq:woneref}
\sup_{\theta_1 \le s \le \tau_1}|W_1(s)-W_1(\theta_1)| 
\;=\; \sup_{\theta_1 \le s \le \tau_1}|\tW_1(s)-\tW_1(\theta_1)| 
\;=\; \frac{1}{2|W_1(\theta_1)|^{a+b-1}R^b} 
\;\le\quad \frac{1}{2R^{f(a-1,b)}}\,.
\end{equation}
Writing \(W_1(s)=(W_1(s)-W_1(\theta_1))+W_1(\theta_1)\) 
and using $R>1$ and $|W_1(\theta_1)| \ge R^{2n}$ as well as \eqref{eq:woneref},
\begin{align}\label{eq:middle}
\sup_{\theta_1 \le s \le \tau_1}\left|\frac{W_1(s)}{W_1(\theta_1)}\right| \quad\le\quad 2\,, 
\qquad 
\sup_{\theta_1 \le s \le \tau_1}\left|\frac{\tW_1(s)}{\tW_1(\theta_1)}\right| \quad\le\quad 2\,.
\end{align}
If $b=0$ then the right-hand side of \eqref{eq:bublu} simplifies, and for $x>2^{a-1}$ it is immediate that
\begin{multline}\label{eq:b=0case}
\Prob{\sup_{\theta_1 \le s \le \tau_1}|\Delta W_1(s)||W_1(s)|^{a-k}|\tW_1(s)|^{k-1}|W_2(s)|^b > x} 
\quad\leq\quad
\\
\Prob{\sup_{\theta_1 \le s \le \tau_1}\left|\frac{W_1(s)}{W_1(\theta_1)}\right|^{a-k}\left|
\frac{\tW_1(s)}{\tW_1(\theta_1)}\right|^{k-1}> x}
\quad=\quad0\,.
\end{multline}
If $b \ge 1$, we can use \eqref{eq:bublu} with \eqref{eq:middle} to obtain 
\begin{multline*}
\Prob{\sup_{\theta_1 \le s \le \tau_1}|\Delta W_1(s)||W_1(s)|^{a-k}|\tW_1(s)|^{k-1}|W_2(s)|^b > x}\\
 \quad \le \quad \Prob{2^{a-1} \sup_{\theta_1 \le s \le \tau_1}
    \left|\frac{W_2(s)}{W_2(\theta_1)}\right|^b > x}
\quad=\quad \Prob{\sup_{\theta_1 \le s \le \tau_1}|W_2(s)| > (x^{1/b}/2^{(a-1)/b})|W_2(\theta_1)|}\\
\quad \quad \le \quad \Prob{\sup_{\theta_1 \le s \le \tau_1}
      |W_2(s)-W_2(\theta_1)| > \left((x^{1/b}/2^{(a-1)/b})-1\right)|W_2(\theta_1)|}
      \,.
\end{multline*}
Now introduce the requirement that $x \ge 2^{a+b-1}$,
so that 
$(x^{1/b}/2^{(a-1)/b})-1
\ge (x/2^{a+b-1})^{1/b}$ for $x \ge 2^{a+b-1}$.
Applying this together with $|W_2(\theta_1)| \ge R^{2n+1}$,
and then applying a Markov inequality argument,
followed by an application of the BDG inequality \citep[p.~163]{KaratzasShreve-2012} 
after conditioning on $\sigma\{(W_1(s), W_2(s)): s \le \theta_1\}$,
\begin{align}\label{eq:bound1bef}
& \Prob{\sup_{\theta_1 \le s \le \tau_1}|\Delta W_1(s)||W_1(s)|^{a-k}|\tW_1(s)|^{k-1}|W_2(s)|^b > x} 
\nonumber\\
& \quad \quad \le \quad \Prob{\sup_{\theta_1 \le s \le \tau_1}|W_2(s)-W_2(\theta_1)| > (x/2^{a+b-1})^{1/b}R^{2n+1}}\nonumber\\
& \quad \quad \le \quad \frac{2^{2(a+b-1)/b}\Expect{\sup_{\theta_1 \le s \le \tau_1}
      |W_2(s)-W_2(\theta_1)|}^2}{x^{2/b}R^{4n+2}}\nonumber\\
&\quad \quad \le \quad  \frac{2^{2(a+b-1)/b}\Expect{\tau_1-\theta_1}}{x^{2/b}R^{4n+2}}\; C
\qquad\text{ when }x \ge 2^{a+b-1}\,.
\end{align}
From \eqref{eq:tautail}, since  $f(a-1, b) \ge 2n + 1$ for $a,b \ge 1$,
\begin{equation}\label{tautheta}
\Expect{\tau_1-\theta_1} \quad\le\quad \frac{C}{R^{2f(a-1,b)}} \quad\le\quad \frac{C}{R^{4n+2}}\,.
\end{equation}
Using this estimate in \eqref{eq:bound1bef}, and using a new constant \(C\),
we obtain the following when \(b\geq1\), when $x \ge 2^{a+b-1}$,
\begin{align}\label{eq:tauflucold}
\Prob{\sup_{\theta_1 \le s \le \tau_1}|\Delta W_1(s)||W_1(s)|^{a-k}|\tW_1(t)|^{k-1}|W_2(s)|^b > x} 
\quad\le\quad 
\frac{2^{2(a+b-1)/b}}{R^{8n+4}x^{2/b}} \; C\,.
\end{align}
Note that \eqref{eq:b=0case} yields an upper bound of \(0\)
when $b=0$ (and $x>2^{a-1}$).
Using \eqref{eq:tauflucold} in \eqref{eq:expan1}, for whatever \(b\), 
and writing \(x=2^{a+b-1} M\) for future convenience of exposition, 
if \(M>1\) then
\begin{align}\label{eq:taufluc}
\Prob{\sup_{\theta_1 \le s \le \tau_1}\left|\frac{({W_1(s)^a} - {\tW_1(s)^a}){W_2(s)^b}}{a 2^{a+b-1}}\right| > M} 
\quad\le\quad 
\begin{cases}
\frac{a}{R^{8n+4}M^{2/b}} \; C & \text{ if } b\geq1\,, 
\\
0                                              & \text{ if }b=0\,.
\end{cases}
\end{align}
We now rewrite \eqref{eq:taufluc} and \eqref{eq:tautail}
to match the first assertion in part (i) of Lemma \ref{lem:tail1}
(after conditioning on $\sigma\{(W_1(s), W_2(s)): s \le \theta_1\}$).
For $s > \theta_1$, we set 
$t = s-\theta_1$,
$B(t)=W_2(t+\theta_1)-W_2(\theta_1)$, 
$Y_t=\frac{({W_1(s)^a} - {\tW_1(s)^a}){W_2(s)^b}}{a2^{a+b-1}}$,
$\tau= \tau_1-\theta_1$, 
$\eps = R^{-1}$.
To match the indices in part (i) of Lemma \ref{lem:tail1}, set
$\alpha = 2/b, \beta =2$ if $b\ge1$, 
and choose any $\beta>0$ if \(b=0\).
Then \eqref{eq:taufluc} is equivalent to the following, holding when \(M>1\):
\[
\Prob{\sup_{0 \le t \le \tau}\left|Y_t\right| > M} 
\quad\le\quad 
\begin{cases}
\frac{a \eps^{8n+4}}{M^{\alpha}} \; C & \text{ if } b\geq1\,, 
\\
0                                              & \text{ if }b=0\,.
\end{cases} 
\]
Note that \(\eps<1\) (since \(R>1\)), so the above implies the weaker inequality,
if \(M>1\) then
\[
 \Prob{\sup_{0 \le t \le \tau}\left|Y_t\right| > M} 
\quad\le\quad 
a C \; {M^{-\alpha}} \,. 
\]

On the other hand \eqref{eq:tautail} becomes
\[
 \Prob{\tau>t \eps^{4n(a+b-1)+2b}} 
 \quad\le\quad C_1 e^{-C_2  t }\,.
\]
Noting \(e^{-C_2  t }\leq 1/(C_2 t)^2\) for \(t>0\), 
and then re-scaling time and using $n \ge 1, a+b > 1$, we obtain
\[
 \Prob{\tau>t} 
 \quad\le\quad
 C_1 e^{-C_2  t / \eps^{4n(a+b-1)+2b} }
 \quad\leq\quad
 (C_1/C_2^2) \eps^{8n(a+b-1)+4b} t^{-2}
 \quad\le\quad
 (C_1/C_2^2)  \eps^2/t^2\,.
\]

We can now apply the first assertion in part (i) of Lemma \ref{lem:tail1} to deduce the following.
For \(z>\eps^{1/4}\),
\[
 \Prob{\left|\int_0^{\tau} Y_s {\d}B_s\right|\geq z}
\quad\leq\quad
\frac{C' \; \eps^{1/(4(b \vee 2))}}{  z^{1/(b \vee 2)} }\,.
 \]
Writing \(Y_s\) in full, this amounts to the following:
when \(x> a 2^{a+b-1}R^{-1/4}\), and taking \(\gamma'=1/(b\vee 2)\),
\begin{align}\label{eq:first-hard-bit}
\Prob{\left|\int_{\theta_1}^{\tau_1} ({W_1(s)^a} - {\tW_1(s)^a}){W_2(s)^b}{\d}W_2(s)\right| >x} 
\quad\le\quad \frac{C''}{(R^{1/4}x)^{\gamma'}}\,.
\end{align}

A similar procedure
leads to a bound concerning 
\(\int_{\theta_1}^{\tau_1} ({W_1(s)^a} - {\tW_1(s)^a}){W_2(s)^{b-1}}{\d}s\).
Here we need only argue for the case \(b\geq1\), as the time integral does not appear for \(I_{(a,0)}\).
Referring to \eqref{eq:taufluc}, but using \(b-1\) instead of \(b\), if \(M>1\) then we obtain
\begin{align}\label{eq:taufluc:b-1}
\Prob{\sup_{\theta_1 \le s \le \tau_1}\left|\frac{({W_1(s)^a} - {\tW_1(s)^a}){W_2(s)^{b-1}}}{a 2^{a+b-2}}\right| > M} 
\quad\le\quad 
\begin{cases}
\frac{a}{R^{8n+4}M^{2/(b-1)}} \; C & \text{ if } b\geq2\,, 
\\
0                                              & \text{ if }b=1\,.
\end{cases}
\end{align}
Choosing $B, \eps$ and $\tau$ as before, 
and again conditioning on $\sigma\{(W_1(s), W_2(s)): s \le \theta_1\}$,
but now setting $Y_t=\frac{({W_1(t)^a} - {\tW_1(t)^a}){W_2(t)^{b-1}}}{a2^{a+b-2}}$.
To match the indices in part (i) of Lemma \ref{lem:tail1}, set
$\alpha = 2/(b-1)$ (for \(b>1\)), and take any $\beta >\alpha$. 
When \(M>1\),
\[
\Prob{\sup_{0 \le t \le \tau}\left|Y_t\right| > M} 
\quad\le\quad 
\begin{cases}
\frac{a \eps^{8n+4}}{M^{\alpha}} \; C & \text{ if } b\geq2\,, 
\\
0                                              & \text{ if }b=1\,.
\end{cases} 
\]
Applying the second assertion in part (i) of Lemma \ref{lem:tail1}, and using
\(\gamma''=\tfrac12(1/(1\vee(b-1)))\), 
\begin{align}\label{eq:second-hard-bit}
\Prob{\left|\int_{\theta_1}^{\tau_1} ({W_1(s)^a} - {\tW_1(s)^a}){W_2(s)^{b-1}}{\d}s\right| >x} 
\quad\le\quad
\frac{C}{(R^{1/2}x)^{\gamma''}} \quad \text{ for } x > a 2^{a+b-2}R^{-1/2}.
\end{align}

Applying the inequalities \eqref{eq:first-hard-bit} and \eqref{eq:second-hard-bit}
to the It\^o representation of $I_{(a,b)}$ given in \eqref{eq:itostrat}, 
we conclude that for any $a \ge 1$, $b \ge 0$, $a+b>1$, 
if $x > a 2^{a+b-1}R^{-1/4}$ then
\begin{align}\label{eq:tauest}
\Prob{|\Delta I_{(a,b)}(\tau_1)-\Delta I_{(a,b)}(\theta_1)|>x} 
\quad\le\quad \frac{C}{(R^{1/4}x)^{\gamma' \wedge \gamma''}}\,.
\end{align}




\medskip\noindent
\textbf{Phase 3: }Now, we address the time interval $[\tau_1, \eta_1]$.
In this phase, 
starting at time $\tau_1$,
synchronous coupling is employed to the driving Brownian motions till $W_2(t+\tau_1)-W_2(\tau_1)$ 
hits the level $-a^{-1}\sgn(\Delta W_1(\tau_1))(\sgn(W_1(\theta_1)))^{a+b-1}$.
Applying the reflection principle to $(W_1(t+\tau_1) - W_1(\tau_1): t \ge 0)$, 
we can deduce the following estimate related to hitting times of Brownian motion:
\begin{align}\label{eq:etatail}
\Prob{\eta_1-\tau_1 >t} \quad\le\quad Ct^{-1/2}\,.
\end{align}
Consider the fluctuations of $\Delta I_{(a,b)}$ on this interval.
Using \eqref{eq:expan1}, it suffices to address the integrals
$$
\int_{\tau_1}^{\eta_1}\Delta W_1(s) W_1(s)^{a-k}\tW_1(s)^{k-1}W_2(s)^b {\d}W_2(s)
$$
and (for \(b\geq1\) and $1 \le k \le a$)
$$
\int_{\tau_1}^{\eta_1}\Delta W_1(s) W_1(s)^{a-k}\tW_1(s)^{k-1}W_2(s)^{b-1} {\d}s\,.
$$
As this is a synchronous coupling phase, 
$|\Delta W_1(t)|=|\Delta W_1(\tau_1)|=\frac{1}{|W_1(\theta_1)|^{a+b-1}R^b}$ for all $t \in [\tau_1, \eta_1]$.
Observe that
\begin{multline*}
\sgn{\Delta W_1(\tau_1)}(\sgn{W_1(\theta_1)})^{a+b-1}\Delta W_1(t)\quad=\quad (\sgn{W_1(\theta_1)})^{a+b-1}|\Delta W_1(t)|\\
\quad=\quad (\sgn{W_1(\theta_1)})^{a+b-1}\frac{1}{|W_1(\theta_1)|^{a+b-1}R^b}
\quad=\quad\frac{1}{W_1(\theta_1)^{a+b-1}R^b} \,.
\end{multline*}
Combining this with the facts that $W_1(\theta_1)=\tW_1(\theta_1)$ and $W_2(\theta_1)=RW_1(\theta_1)$,
if $t \in [\tau_1, \eta_1]$ then
\begin{multline}\label{eq:etafluc1}
\sgn{\Delta W_1(\tau_1)}(\sgn{W_1(\theta_1)})^{a+b-1} \times \Delta W_1(t)W_1(t)^{a-k}\tW_1(t)^{k-1}{W_2(t)^b}\\
\quad = \quad
\left(\frac{W_1(t)}{W_1(\theta_1)}\right)^{a-k}\left(\frac{\tW_1(t)}{\tW_1(\theta_1)}\right)^{k-1}\left(\frac{W_2(t)}{W_2(\theta_1)}\right)^b.
\end{multline}
Set $A_1(t)=\frac{W_1(t)}{W_1(\theta_1)}$, $\tA_1(t)=\frac{\tW_1(t)}{\tW_1(\theta_1)}$ and $A_2(t)=\frac{W_2(t)}{W_2(\theta_1)}$.
Observe that for $1 \le k \le a$
\begin{multline}\label{eq:abound1}
|A_1(t)^{a-k}\tA_1(t)^{k-1}A_2(t)^b-1| \quad\le\quad
\\
|A_1(t)^{a-k}-1||\tA_1(t)^{k-1}||A_2(t)^b| + |\tA_1(t)^{k-1}-1||A_2(t)^b| + |A_2(t)^b-1|\,.
\end{multline}
We will show that the first term above is small with high probability.
If $a=1$, or more generally if $k=a$, then the first term is identically zero.
If $a \ge 2$ and $k \le a-1$ then
\begin{align}\label{eq:a1bound}
|A_1(t)^{a-k}-1||\tA_1(t)^{k-1}||A_2(t)^b| 
\quad\le\quad
\sum_{j=1}^{a-k}|A_1(t)-1||A_1(t)|^{a-k-j}|\tA_1(t)|^{k-1}|A_2(t)|^b\,.
\end{align}
Fix $x \ge 1/R^{2n}$. Recall that \(|W_1(\theta_1)|\geq R^{2n}\),
and note firstly that for $x \ge 1/R^{2n}$, by the reflection coupling implications summarized in \eqref{eq:woneref},
$$
\Prob{|W_1(\tau_1) - W_1(\theta_1)| > xR^{2n}/2} \quad \le\quad \Prob{|W_1(\tau_1) - W_1(\theta_1)| > 1/2}=0\,.
$$
and secondly
by a Tchebychev inequality argument and Doob's $L^2$-maximal inequality
$$
\Prob{\sup_{\tau_1 \le t \le \tau_1 + T}|W_1(t)-W_1(\tau_1)| \quad>\quad x R^{2n}/2} \le \frac{CT}{x^2R^{4n}}\,.
$$
Then
\begin{multline}\label{eq:later}
\Prob{\sup_{\tau_1 \le t \le \eta_1}|A_1(t)-1| > x} 
\quad\le\quad \Prob{\sup_{\tau_1 \le t \le \eta_1}|W_1(t)-W_1(\theta_1)| > x R^{2n}}
\\
\quad\le\quad \Prob{\eta_1 - \tau_1 >T} + \Prob{\sup_{\tau_1 \le t \le \tau_1 + T}|W_1(t)-W_1(\theta_1)| > x R^{2n}}
\\
\quad\le\quad \Prob{\eta_1 - \tau_1 >T} + \Prob{\sup_{\tau_1 \le t \le \tau_1 + T}|W_1(t)-W_1(\tau_1)| > x R^{2n}/2}
\\
\quad \quad \quad + \Prob{|W_1(\tau_1) - W_1(\theta_1)| > xR^{2n}/2}
\quad\le\quad \frac{C}{\sqrt{T}} + \frac{CT}{x^2R^{4n}} 
\quad\le\quad \frac{C}{(xR^{2n})^{2/3}}\,,
\end{multline}
where the last inequality follows by taking $T=(xR^{2n})^{4/3}$.

Similarly, for $x \ge 2$,
\begin{multline}\label{eq:bound2}
\Prob{\sup_{\tau_1 \le t \le \eta_1}|A_1(t)| > x} \quad=\quad 
\Prob{\sup_{\tau_1 \le t \le \eta_1}|W_1(t)| > x|W_1(\theta_1)|}
\\
\quad\le\quad \Prob{\sup_{\tau_1 \le t \le \eta_1}|W_1(t)-W_1(\theta_1)| > x|W_1(\theta_1)|/2}
\\
\quad\le\quad \Prob{\sup_{\tau_1 \le t \le \eta_1}|W_1(t)-W_1(\theta_1)| > xR^{2n}/2} \le \frac{C}{(xR^{2n})^{2/3}}\,,
\end{multline}
where the last inequality follows from the computations performed to obtain \eqref{eq:later}.
A similar estimate for $\Prob{\sup_{\tau_1 \le t \le \eta_1}|\tA_1(t)| > x}$ 
holds by replacing $W_1$ with $\tW_1$ in the above calculations. To derive an analogous estimate for $\Prob{\sup_{\tau_1 \le t \le \eta_1}|A_2(t)| > x}$, first observe that
$$
\Expect{|W_2(\tau_1)-W_2(\theta_1)|}^2 = \Expect{\tau_1 - \theta_1} \le C/R^{4n+2},
$$
where the first equality is because conditional on $\sigma\{(W_1(s), W_2(s)): s \le \theta_1\}$, $W_2 - W_2(\theta_1)$ is independent of $\tau_1 - \theta_1$ and the last inequality follows from \eqref{tautheta}. Using this observation along with the Tchebychev inequality, we obtain
\begin{multline}\label{newbd1}
\Prob{\sup_{\tau_1 \le t \le \eta_1}|W_2(t)-W_2(\theta_1)| > \frac{xR^{2n+1}}{2}}\\
\quad\le\quad \Prob{\eta_1 - \tau_1 >T} + \Prob{\sup_{\tau_1 \le t \le \tau_1 + T}|W_2(t)-W_2(\theta_1)| > \frac{xR^{2n+1}}{2}}
\\
\quad\le\quad \Prob{\eta_1 - \tau_1 >T} + \Prob{\sup_{\tau_1 \le t \le \tau_1 + T}|W_2(t)-W_2(\tau_1)| > \frac{xR^{2n+1}}{4}}\\
 + \Prob{|W_2(\tau_1) - W_2(\theta_1)| > \frac{xR^{2n+1}}{4}}\\
\quad\le\quad \frac{C}{\sqrt{T}} + \frac{CT}{x^2R^{4n+2}}  + \frac{C}{x^2R^{8n+4}}
\quad\le\quad \frac{C}{(xR^{2n})^{2/3}}\,,
\end{multline}
where the last inequality follows by taking $T=(xR^{2n})^{4/3}$. Using \eqref{newbd1} and recalling $|W_2(\theta_1)| \ge R^{2n+1}$,
\begin{multline}\label{eq:boundnew}
\Prob{\sup_{\tau_1 \le t \le \eta_1}|A_2(t)| > x} 
\quad\le\quad \Prob{\sup_{\tau_1 \le t \le \eta_1}|W_2(t)-W_2(\theta_1)| > x|W_2(\theta_1)|/2}
\\
\quad\le\quad \Prob{\sup_{\tau_1 \le t \le \eta_1}|W_2(t)-W_2(\theta_1)| > xR^{2n+1}/2}\le \frac{C}{(xR^{2n})^{2/3}}.
\end{multline}
From the above estimates, we can argue the following in case $x \ge 2^{2(a+b-1)}/R^{2n}$:
\begin{multline*}
\Prob{\sup_{\tau_1 \le t \le \eta_1}|A_1(t)-1||A_1(t)|^{a-k-j}|\tA_1(t)|^{k-1}|A_2(t)|^b >x}\\
\quad \le \quad\Prob{\sup_{\tau_1 \le t \le \eta_1}|A_1(t)-1| > \frac{\sqrt{x}}{\sqrt{R^{2n}}}} + \Prob{\sup_{\tau_1 \le t \le \eta_1}|A_1(t)| > (\sqrt{R^{2n}x})^{\frac{1}{a+b-1-j}}}\\
+\Prob{\sup_{\tau_1 \le t \le \eta_1}|\tA_1(t)| > (\sqrt{R^{2n}x})^{\frac{1}{a+b-1-j}}} + \Prob{\sup_{\tau_1 \le t \le \eta_1}|A_2(t)| > (\sqrt{R^{2n}x})^{\frac{1}{a+b-1-j}}}\\
\quad\le\quad C\frac{1}{(R^{2n}x)^{\gamma}}
\end{multline*}
for some $\gamma>0$ (in fact 
$\gamma=1/3$) 
that does not depend on $R$ (the last three probabilities appearing after the first inequality above can be taken to be zero if $a+b-1-j=0$).

By applying the above argument to each term on the right hand side of \eqref{eq:a1bound}, we obtain
$$
\Prob{\sup_{\tau_1 \le t \le \eta_1}|A_1(t)^{a-k}-1||\tA_1(t)^{k-1}||A_2(t)^b|>x} \le \frac{C}{(R^{2n}x)^{\gamma}}  \quad \text{ for } x \ge (a-k) 2^{2(a+b-1)}/R^{2n}.
$$
The terms $|\tA_1(t)^{k-1}-1||A_2(t)^b|$ and $|A_2(t)^b-1|$ appearing in \eqref{eq:abound1} are
subject to estimates of the same form,
based on $\Prob{\sup_{\tau_1 \le t \le \eta_1}|\tA_1(t)-1| > x}$ and $\Prob{\sup_{\tau_1 \le t \le \eta_1}|A_2(t)-1| > x}$ respectively
in place of $\Prob{\sup_{\tau_1 \le t \le \eta_1}|A_1(t)-1| > x}$,
but otherwise using the same arguments.
Hence \eqref{eq:abound1} and the above estimates yield the following for $x \ge (a+b-1)2^{2(a+b-1)}/R^{2n}$:
\begin{align}\label{eq:tuk}
\Prob{\sup_{\tau_1 \le t \le \eta_1}|A_1(t)^{a-k}\tA_1(t)^{k-1}A_2(t)^b-1|>x} \quad\le\quad \frac{C}{(R^{2n}x)^{\gamma}}\,.
\end{align}
Thus \eqref{eq:etafluc1} yields (when $x \ge (a+b-1)2^{2(a+b-1)}/R^{2n}$)
\begin{multline*}
\Prob{\sup_{\tau_1 \le t \le \eta_1}\left|\Delta W_1(t) W_1(t)^{a-k}\tW_1(t)^{k-1}{W_2(t)^b}
  -\sgn{\Delta W_1(\tau_1)}(\sgn{W_1(\theta_1)})^{a+b-1}\right|>x}\\ 
  \quad\le \quad\frac{C}{(R^{2n}x)^{\gamma}}\,.
\end{multline*}
The above holds for all $1 \le k \le a$; consequently \eqref{eq:expan1} implies that, for $x \ge a(a+b-1)2^{2(a+b-1)}/R^{2n}$,
\begin{multline}\label{eq:etafluc2}
\Prob{\sup_{\tau_1 \le t \le \eta_1}\left|({W_1(t)^a} - {\tW_1(t)^a}){W_2(t)^b}
-a\sgn{\Delta W_1(\tau_1)}(\sgn{W_1(\theta_1)})^{a+b-1}\right|>x}\\
\quad\le\quad \frac{C}{(R^{2n}x)^{\gamma}}\,.
\end{multline}

Now $\Prob{\eta_1-\tau_1>t} \le Ct^{-1/2}$;
so the first assertion in part (ii) of Lemma \ref{lem:tail1} implies there is $\gamma'>0$, not depending on $R$, such that
for $x \ge a(a+b-1)2^{2(a+b-1)}/R^{2n}$
\begin{multline*}
\mathbb{P}\Big[
\Big|
\int_{\tau_1}^{\eta_1}({W_1(s)^a} - {\tW_1(s)^a}){W_2(s)^b}{\d}W_2(s)\\
-a\sgn{\Delta W_1(\tau_1)}(\sgn{W_1(\theta_1)})^{a+b-1}(W_2(\eta_1)-W_2(\tau_1))
\Big|
>x
\Big]
\\
\quad\le\quad \frac{C}{(R^{2n}x)^{\gamma'}}\,.
\end{multline*}
But it follows from the definition of $\eta_1$ that
$$
W_2(\eta_1)-W_2(\tau_1)=-a^{-1}\sgn(\Delta W_1(\tau_1))(\sgn(W_1(\theta_1)))^{a+b-1}\,.
$$
Together with the above inequality this yields, for $x \ge a(a+b-1)2^{2(a+b-1)}/R^{2n}$,
\begin{align}\label{eq:etaend1}
\Prob{
\left|
\int_{\tau_1}^{\eta_1}({W_1(s)^a} - {\tW_1(s)^a}){W_2(s)^b}{\d}W_2(s)+1
\right|>x}
\quad\le\quad \frac{C}{(R^{2n}x)^{\gamma}}\,.
\end{align}

To estimate the integral $\int_{\tau_1}^{\eta_1}\Delta W_1(s) W_1(s)^{a-k}\tW_1(s)^{k-1}W_2(s)^{b-1}{\d}s$ for $b \ge 1$, we can once more use the synchronous coupling of $W_1$, $\tW_1$ on $[\tau_1, \eta_1]$ to show that for any $t \in [\tau_1, \eta_1]$,
$$
|\Delta W_1(t) W_1(t)^{a-k}\tW_1(t)^{k-1}W_2(t)^{b-1}| 
\quad\le\quad
\frac{1}{R^{2n+1}} \left|\frac{W_1(t)}{W_1(\theta_1)}\right|^{a-k}
\left|\frac{\tW_1(t)}{\tW_1(\theta_1)}\right|^{k-1}
\left|\frac{W_2(t)}{W_2(\theta_1)}\right|^{b-1}
\,.
$$
For $b \ge 1$ we may use \eqref{eq:tuk} to show, for $x \ge 2^{2(a+b-1)}$,
\begin{align*}
\Prob{\sup_{\tau_1 \le t \le \eta_1}\left|\frac{W_1(t)}{W_1(\theta_1)}\right|^{a-k}
\left|\frac{\tW_1(t)}{\tW_1(\theta_1)}\right|^{k-1}
\left|\frac{W_2(t)}{W_2(\theta_1)}\right|^{b-1} >x}
\quad\le\quad
\frac{C}{(R^{2n}x)^{\gamma}}\,.
\end{align*}
Thus, for $x \ge 2^{2(a+b-1)}/R^{2n+1}$,
\begin{align*}
\Prob{\sup_{\tau_1 \le t \le \eta_1}|\Delta W_1(s) W_1(s)^{a-k}\tW_1(s)^{k-1}W_2(s)^{b-1}| >x} 
\quad\le\quad \frac{C}{(R^{4n+1}x)^{\gamma}}\,.
\end{align*}
Using the above and the fact that $\Prob{\eta_1-\tau_1>t} \le Ct^{-1/2}$ in the second assertion in part (ii) of Lemma \ref{lem:tail1}, we obtain $\gamma, \delta>0$ not depending on $R$ such that
\begin{align*}
\Prob{\left|\int_{\tau_1}^{\eta_1}\Delta W_1(s) W_1(s)^{a-k}\tW_1(s)^{k-1}W_2(s)^{b-1}{\d}s\right| > x} \le \frac{C}{(R^{\delta}x)^{\gamma}}
\end{align*}
for $x \ge (a+b-1)2^{2(a+b-1)}R^{-\delta}$.
But then we can use \eqref{eq:expan1} to deduce
\begin{align}\label{eq:etaend2}
\Prob{\left|\int_{\tau_1}^{\eta_1}({W_1(s)^a} - {\tW_1(s)^a}){W_2(s)^{b-1}}{\d}s\right|>x} 
\quad\le\quad \frac{C}{(R^{\delta}x)^{\gamma}}
\end{align}
for $x \ge a(a+b-1)2^{2(a+b-1)}R^{-\delta}$.
Recalling the expression of $I_{(a,b)}$ in terms of the It\^o integral and the time integral given in \eqref{eq:itostrat}, we obtain from \eqref{eq:etaend1} and \eqref{eq:etaend2},
\begin{align}\label{eq:etaend}
\Prob{\left|\Delta I_{(a,b)}(\eta_1)-\Delta I_{(a,b)}(\tau_1) + 1\right|>x} 
\quad\le\quad \frac{C}{(R^{\delta}x)^{\gamma}}
\end{align}
for $x \ge a(a+b-1)2^{2(a+b-1)}R^{-\delta}$.

\medskip\noindent
\textbf{Phase 4: }
The next phase occurs in the time interval $[\eta_1, \lambda_1]$.
In this phase, after time $\eta_1$ the Brownian motions $W_1$ and $\tW_1$ are subjected to reflection coupling till they meet.
Applying the reflection principle,  and using the fact that \(|\Delta W(\eta_1)|=1/(|W(\theta_1)|^{a+b-1}R^b)\)
together with other consequences of the definitions of the stopping times \(\theta_1\) and \(\tau_1\), we see that when \(t>0\)
\begin{equation}\label{eq:lambdatail}
\Prob{\lambda_1-\eta_1>t/R^{2f(a-1,b)}} \quad\le\quad Ct^{-1/2}\,.
\end{equation}
Once again \eqref{eq:expan1} can be applied, so it suffices to consider the integrals
$$
\int_{\eta_1}^{\lambda_1}\Delta W_1(s) W_1(s)^{a-k}\tW_1(s)^{k-1}W_2(s)^b {\d}W_2(s)
$$
and
$$
\int_{\eta_1}^{\lambda_1}\Delta W_1(s) W_1(s)^{a-k}\tW_1(s)^{k-1}W_2(s)^{b-1} {\d}s
$$
for $1 \le k \le a$.

For $\eta_1 \le t \le \lambda_1$ we can write
\begin{multline}\label{eq:tutu}
|\Delta W_1(t)| |W_1(t)|^{a-k}|\tW_1(t)|^{k-1}|W_2(t)|^b \quad=\quad
\\
\left(|\Delta W_1(t)||W_1(\theta_1)|^{a+b-1}R^b\right)
  \left|\frac{W_1(t)}{W_1(\theta_1)}\right|^{a-k}\left|\frac{\tW_1(t)}{\tW_1(\theta_1)}\right|^{k-1}\left|\frac{W_2(t)}{W_2(\theta_1)}\right|^b
\,.
\end{multline}
Recalling that $|\Delta W_1(\eta_1)|=|\Delta W_1(\tau_1)|=\frac{1}{|W_1(\theta_1)|^{a+b-1} R^b}$,
and bearing in mind that $W_1$ and $\tW_1$ are reflection coupled on $[\eta_1, \lambda_1]$, when $x \ge 1$ it follows that
\begin{multline}\label{eq:l1}
\Prob{\sup_{\eta_1 \le t \le \lambda_1} |\Delta W_1(t)||W_1(\theta_1)|^{a+b-1} R^b \ge x} \\
\quad=\quad\Prob{\text{Brownian motion starting from } \frac{1}{2} \text{ hits } \frac{x}{2} \text{ before zero}}
\quad=\quad\frac{1}{x}\,,
\end{multline}
where the last equality follows from the optional stopping theorem.
Fixing $x \ge 2$,
we can employ \eqref{eq:lambdatail} and a Tchebychev inequality argument to show
\begin{multline*}
\Prob{\sup_{\eta_1 \le t \le \lambda_1}\left|\frac{W_1(t)}{W_1(\theta_1)}\right|>x} 
\quad\le\quad \Prob{\sup_{\eta_1 \le t \le \lambda_1}|W_1(t)-W_1(\theta_1)|>xR^{2n}/2}
\\
\quad\le\quad \Prob{\sup_{\eta_1 \le t \le \lambda_1}|W_1(t)-W_1(\eta_1)|>xR^{2n}/4} + \Prob{|W_1(\eta_1)-W_1(\theta_1)|>xR^{2n}/4}
\\
\quad \le\quad \Prob{\lambda_1 - \eta_1> T} + \Prob{\sup_{\eta_1 \le t \le \eta_1 + T}|W_1(t)-W_1(\eta_1)|>xR^{2n}/4}
\\
\qquad\qquad + \Prob{|W_1(\eta_1)-W_1(\theta_1)|>xR^{2n}/4}
\\
\quad\le\quad \frac{C}{R^{f(a-1,b)}\sqrt{T}} + \frac{CT}{x^2R^{4n}} + \frac{C}{(xR^{2n})^{2/3}}.
\end{multline*}
The bound $\Prob{|W_1(\eta_1)-W_1(\theta_1)|>xR^{2n}/4}\le \frac{C}{(xR^{2n})^{2/3}}$ follows from the calculations leading to \eqref{eq:later}, 
where, in fact,  we obtained the following bound when \(x \ge \frac{1}{R^{2n}}\):
$$
\Prob{\sup_{\tau_1 \le t \le \eta_1}|W_1(t)-W_1(\theta_1)| > x R^{2n}}
\quad\le\quad \frac{C}{(xR^{2n})^{2/3}}\,.
$$
Taking $T=(xR^{2n})^{4/3}$, we obtain the following when \(x\ge2\):
\begin{align}\label{eq:l2}
\Prob{\sup_{\eta_1 \le t \le \lambda_1}\left|\frac{W_1(t)}{W_1(\theta_1)}\right|>x} 
\quad\le\quad
\frac{C}{(xR^{2n})^{2/3}}\,.
\end{align}
Similar estimates for $\Prob{\sup_{\eta_1 \le t \le \lambda_1}\left|\frac{\tW_1(t)}{\tW_1(\theta_1)}\right|>x}$ 
and $\Prob{\sup_{\eta_1 \le t \le \lambda_1}\left|\frac{W_2(t)}{W_2(\theta_1)}\right|>x}$ 
follow by replacing $W_1$ by $\tW_1$ and $W_2$ respectively in the above calculations (in the latter case, we use \eqref{newbd1}).
Using these estimates along with \eqref{eq:l2} and \eqref{eq:tutu}, we obtain for $x \ge 2^{a+b}$,
\begin{multline}\label{eq:l3}
\Prob{\sup_{\eta_1 \le t \le \lambda_1}|\Delta W_1(t)| |W_1(t)|^{a-k}|\tW_1(t)|^{k-1}|W_2(t)|^b > x}
\\ 
\quad=\quad \Prob{ \left(|\Delta W_1(t)||W_1(\theta_1)|^{a+b-1}R^b\right)\left|\frac{W_1(t)}{W_1(\theta_1)}\right|^{a-k}\left|\frac{\tW_1(t)}{\tW_1(\theta_1)}\right|^{k-1}\left|\frac{W_2(t)}{W_2(\theta_1)}\right|^b > x}
\\
\quad\le\quad \Prob{\sup_{\eta_1 \le t \le \lambda_1} |\Delta W_1(t)||W_1(\theta_1)|^{a+b-1} R^b \ge x^{\frac{1}{a+b}}} + \Prob{\sup_{\eta_1 \le t \le \lambda_1}\left|\frac{W_1(t)}{W_1(\theta_1)}\right|>x^{\frac{1}{a+b}}}
\\
\quad\quad\qquad + \Prob{\sup_{\eta_1 \le t \le \lambda_1}\left|\frac{\tW_1(t)}{\tW_1(\theta_1)}\right|>x^{\frac{1}{a+b}}} + \Prob{\sup_{\eta_1 \le t \le \lambda_1}\left|\frac{W_2(t)}{W_2(\theta_1)}\right|>x^{\frac{1}{a+b}}}
\\
\quad\le\quad \frac{1}{x^{\frac{1}{a+b}}} + \frac{C}{(x^{\frac{1}{a+b}}R^{2n})^{2/3}} 
\quad\le\quad Cx^{-\frac{2}{3(a+b)}}\,,
\end{multline}
where the last step follows as $R>1$.
From \eqref{eq:lambdatail}, \eqref{eq:l3} and the first assertion in part (i) of Lemma \ref{lem:tail1}, 
it follows that there are $\delta, \gamma >0$ not depending on $R$ such that
\begin{align*}
\Prob{\left|\int_{\eta_1}^{\lambda_1}({W_1(s)^a} - {\tW_1(s)^a}){W_2(s)^b}{\d}W_2(s)\right|>x}
\quad\le\quad \frac{C}{(R^{\delta}x)^{\gamma}}
\end{align*}
for $x \ge 2^{a+b}/R^{\delta}$.
Arguing as above, using \eqref{eq:lambdatail}, and \eqref{eq:l3} but with $b-1$ replacing $b$, 
and appealing to the second assertion in part (i) of Lemma \ref{lem:tail1}, if $b \ge 1$ then
\begin{align*}
\Prob{\left|\int_{\eta_1}^{\lambda_1}\Delta W_1(s) W_1(s)^{a-k}\tW_1(s)^{k-1}W_2(s)^{b-1} {\d}s\right| >x}
\quad\le\quad \frac{C}{(R^{\delta}x)^{\gamma}}
\end{align*}
for $x \ge 2^{a+b-1}/R^{\delta}$.
From the above two bounds, if $x \ge 2^{a+b}/R^{\delta}$ then
\begin{align}\label{eq:lambdaend}
\Prob{|\Delta I_{(a,b)}(\lambda_1)- \Delta I_{(a,b)}(\eta_1)| >x}
\quad\le\quad
\frac{C}{(R^{\delta}x)^{\gamma}}\,.
\end{align}

\medskip\noindent
\textbf{Phase 5: }
The final phase concerns the interval $[\lambda_1, \beta_1]$,
in which the Brownian motions $(W_1,W_2)$ and $(\tW_1, \tW_2)$ 
are coupled synchronously till the time $\beta_1$ when $(W_1,W_2)=\tW_1, \tW_2)$ hits the line $u_2=Ru_1$.
Since $(W_1,W_2)=(\tW_1, \tW_2)$ during this time interval, $\Delta I_{(a,b)}(\beta_1)=\Delta I_{(a,b)}(\lambda_1)$.
We claim there is a positive constant $C$ not depending on $R$ such that
\begin{align}\label{eq:betatail}
\Prob{\beta_1 - \lambda_1 >t} \quad\le\quad Ct^{-1/6}\,.
\end{align}
To see this, observe that $\beta_1-\lambda_1$ depends on how far away the Brownian motion $(W_1, W_2)$ is from the line $u_2=Ru_1$ at time $\lambda_1$.
As $W_2(\theta_1)=RW_1(\theta_1)$, this distance, in turn, depends on the size of the total duration $\lambda_1-\theta_1$ of the previous three phases.
Indeed, for any $\alpha, x >0$ (to be chosen later),
\begin{align*}
\Prob{\beta_1 - \lambda_1 >t} 
\quad&\le\quad \Prob{\beta_1 - \lambda_1 >t, \; \lambda_1-\theta_1 \le t^{\alpha}} + \Prob{\lambda_1-\theta_1 > t^{\alpha}}\\
\quad&\le\quad \Prob{|(W_1,W_2)(\lambda_1) - (W_1,W_2)(\theta_1)| >x, \; \lambda_1-\theta_1 \le t^{\alpha}}\\
\quad&\qquad + \Prob{|(W_1,W_2)(\lambda_1) - (W_1,W_2)(\theta_1)| \le x, \; \beta_1 - \lambda_1 >t}\\ \quad&\qquad + \Prob{\lambda_1-\theta_1 > t^{\alpha}}
\,.
\end{align*}
To estimate the first probability above, note that an application of the strong Markov property at time $\theta_1$
allows us to deduce
\begin{multline*}
\Prob{|(W_1,W_2)(\lambda_1) - (W_1,W_2)(\theta_1)| >x, \; \lambda_1-\theta_1 \le t^{\alpha}}\\
\quad\le\quad \Prob{\sup_{s \in[\theta_1, \theta_1 + t^{\alpha}]}|(W_1,W_2)(s) - (W_1,W_2)(\theta_1)| >x}
\quad\le\quad C\frac{t^{\alpha/2}}{x}\,,
\end{multline*}
where the last inequality follows from Doob's submartingale inequality applied to the radial part of two-dimensional Brownian motion.
The second probability is controlled by conditioning on the past event $[|(W_1,W_2)(\lambda_1) - (W_1,W_2)(\theta_1)| \le x]$ 
and using the strong Markov property to argue that the hitting time on the line $u_2=Ru_1$ by the Brownian motion 
$((W_1,W_2)(t) - (W_1,W_2)(\lambda_1): t \ge \lambda_1)$ is stochastically dominated by 
the hitting time on zero by a one dimensional Brownian motion starting from $x$.
Therefore,
$$
\Prob{|(W_1,W_2)(\lambda_1) - (W_1,W_2)(\theta_1)| \le x, \; \beta_1 - \lambda_1 >t} 
\quad\le\quad C\frac{x}{\sqrt{t}}\,.
$$
From \eqref{eq:tautail}, \eqref{eq:etatail} and \eqref{eq:lambdatail}, we deduce that
\begin{align}\label{eq:tailhalf}
\Prob{\lambda_1-\theta_1 > t^{\alpha}} 
\quad\le\quad
Ct^{-\alpha/2}\,.
\end{align}
Putting these bounds together, it follows that
\begin{align*}
\Prob{\beta_1 - \lambda_1 >t} 
\quad\le\quad
C\frac{t^{\alpha/2}}{x} + C\frac{x}{\sqrt{t}} + Ct^{-\alpha/2}\,.
\end{align*}
The target inequality \eqref{eq:betatail} is obtained by taking $\alpha=1/3$ and $x = t^{1/3}$ in the above bound.

\noindent
From \eqref{eq:tauest}, \eqref{eq:etaend} and \eqref{eq:lambdaend}, we see that there exist positive constants $C_1, C_2, \delta, \gamma$ not depending on $R, w, i$ such that
\begin{align}\label{eq:coupleclose1}
\Prob{\left|\Delta I_{(a,b)}(\beta_1)\right|>x}
\quad\le\quad
\frac{C_1}{(R^{\delta}x)^{\gamma}} \quad \quad \text{ for } x \ge C_2/R^{\delta}.
\end{align}

\medskip\noindent
\textbf{B: Describing subsequent cycles and successful coupling}\\
The above account gives a description of the five phases that constitute the first cycle.
Subsequent cycles are defined similarly as follows:

For $t \ge \beta_1$, we apply scaling using Lemma \ref{lem:scaling} with $r= \left|\Delta I_{(a,b)}(\beta_1)\right|^{-1/(a+b+1)}$ 
and define further stopping times $\theta_2, \dots, \beta_2$ corresponding to $\theta_1, \dots, \beta_1$ for the scaled process, 
and continue in this fashion to obtain successive cycles.
As in the proof of Lemma \ref{lem:Heisenberg}, in order to show that constructing these cycles leads to a successful coupling
we need to show that $\lim_{k \rightarrow \infty}\Delta I_{(a,b)}(\beta_k) = 0$, and $\lim_{k \rightarrow \infty}\beta_k < \infty$ almost surely. 
This would imply that the end points of these cycles have an accumulation point and thus that the coupling is successful in finite time.
We now demonstrate that these facts follow from the estimates obtained above, \emph{via} Lemma \ref{lem:tail2}.

For $k \ge 1$, if $|\Delta I_{(a,b)}(\beta_{k-1})|=0$, then the coupling is successful.
If the coupling is not successful, define $\displaystyle{X_k=\frac{R^{\delta}|\Delta I_{(a,b)}(\beta_k)|}{|\Delta I_{(a,b)}(\beta_{k-1})|}}$, 
where $\delta$ is as used in \eqref{eq:coupleclose1} and we adopt the convention that $\beta_0=0$.
Taking $\tau_k=1$ for $k \ge 1$, we see that $X_k$, $\tau_k$ satisfy the hypotheses of Lemma \ref{lem:tail2}, and thus we obtain $R'_0>1$ such that for all $R \ge R'_0$, 
$$
\sum_{k=1}^{\infty}R^{-k\delta}\left(\Pi_{j=1}^{k}X_j\right) \quad<\quad \infty
$$
almost surely. In particular this implies that almost surely
$$
\lim_{k \rightarrow \infty} R^{-k\delta}\left(\Pi_{j=1}^{k}X_j\right)
\quad=\quad \lim_{k \rightarrow \infty}|\Delta I_{(a,b)}(\beta_k)|
\quad=\quad0\,.
$$

Choose and fix any $R \ge R'_0$.
 From \eqref{eq:thetatail}, \eqref{eq:tautail}, \eqref{eq:etatail}, \eqref{eq:lambdatail} and \eqref{eq:betatail}, we have $\alpha>0$ such that
\begin{equation}\label{eq:fin}
\Prob{\beta_1>R^{4n+2}t} \quad\le\quad Ct^{-\alpha}
\end{equation}
for $t \ge 1$.
Write $\mathcal{F}_k = \sigma\{(W_1(s), W_2(s))\, : \, s \le \beta_k\}$.
Define $\displaystyle{\tau^*_k=\frac{|\beta_k-\beta_{k-1}|}{\left|\Delta I_{(a,b)}(\beta_{k-1})\right|^{2/(a+b+1)}}}$ for $k \ge 1$. By \eqref{eq:fin}, $\tau^*_k$ satisfies
$$
\Prob{\tau^*_{k+1} > R^{4n+2} t \mid \mathcal{F}_k} \quad\le\quad C t^{-\alpha}\,.
$$
for $t \ge 1$.
Define $\displaystyle{X^*_k=\frac{R^{2\delta/(a+b+1)}|\Delta I_{(a,b)}(\beta_k)|^{2/(a+b+1)}}{|\Delta I_{(a,b)}(\beta_{k-1})|^{2/(a+b+1)}}}$ for $k \ge 1$, where $\delta$ is the same as that used in \eqref{eq:coupleclose1}.
By \eqref{eq:coupleclose1}, observe that
$$
\Prob{X^*_{k+1} > x \mid \mathcal{F}_k} \quad\le\quad Cx^{-\gamma (a+b+1)/2} \quad \text{ for } x \ge 1.
$$
The following holds:
$$
\beta_{k+1}\quad=\quad \tau^*_1+\sum_{l=1}^{k}R^{-2l\delta/(a+b+1)}\left(\Pi_{j=1}^{l}X^*_j\right)\tau^*_{l+1}\,.
$$
Thus, for any $\gamma' < \alpha \wedge \frac{\gamma (a+b+1)}{2}$, 
using Lemma \ref{lem:tail2} with $(X^*_i, \tau^*_i/R^{4n+2})$ in place of $(X_i, \tau_i)$, we obtain $R''_0 \ge R'_0$ such that for every $R \ge R''_0$,
\begin{align*}
\Prob{\tau^*_1+\sum_{l=1}^{\infty}R^{-2l\delta/(a+b+1)}\left(\Pi_{j=1}^{l}X^*_j\right)\tau^*_{l+1}>R^{4n+2}t} \quad\le\quad Ct^{-\gamma'}\,.
\end{align*}
This shows that the coupling construction represented by $\mathbb{P}_R$ 
yields an almost surely successful coupling with coupling time given by $T_{R,(a,b)}=\lim_{k \rightarrow \infty} \beta_k$. $R_0$ claimed in the theorem can be taken to be $R''_0$.

From the coupling construction, we see that the \emph{active region} $S_{R,(a,b)}$ referred to in the theorem can be written as
\begin{equation*}
S_{R,(a,b)} \quad=\quad\bigcup_{k=1}^{\infty}[\theta_k,\lambda_k]\,.
\end{equation*}
The estimate on the tail probabilities of $|S_{R,(a,b)}|$,
claimed in the statement of the lemma, 
follows from Lemma \ref{lem:tail2} using an argument similar to that given above, 
after re-scaling by considering $\frac{|\lambda_k-\theta_k|}{\left|\Delta I_{(a,b)}(\beta_{k-1})\right|^{2/(a+b+1)}}$ for $\tau^*_k$ 
(in fact, it follows from \eqref{eq:tailhalf} that the tail estimate holds for any $\gamma<1/2$).

Assertion (ii) claimed in the lemma follows first from observing that
$$
\sup_{t \le T_{R, (a,b)}} | \Delta W_1(t)| \quad\le\quad \sum_{l=1}^{\infty}\sup_{t \in [\theta_l, \lambda_l]} | \Delta W_1(t)| 
$$
and then from applying Lemma \ref{lem:tail2} with $(\sqrt{X^*_k}, M^*_k)$ in place of $(X_k, \tau_k)$, where
$$
M^*_k\quad=\quad
R^{f(a-1,b)}\frac{\sup_{t \in [\theta_k, \lambda_k]} | \Delta W_1(t)|}{|\Delta I_{(a,b)}(\beta_{k-1})|^{1/(a+b+1)}}\,.
$$
The tail estimates for $M^*_k$ needed to apply Lemma \ref{lem:tail2} are derived by recalling $|W_1(\theta_1)| \ge R^{2n}$ and applying scaling to deduce for $x \ge 1$
\begin{multline*}
\Prob{M^*_{k+1} > x \mid \mathcal{F}_k} \le \Prob{R^{f(a-1,b)}\sup_{t \in [\theta_1, \lambda_1]} | \Delta W_1(t)|> x}\\
= \Prob{R^{f(a-1,b)}\sup_{t \in [\eta_1, \lambda_1]} | \Delta W_1(t)|> x}
\le  \Prob{\sup_{t \in [\eta_1, \lambda_1]} | \Delta W_1(t)||W_1(\theta_1)|^{a+b-1}R^b> x} =x^{-1},
\end{multline*}
where the last step follows from \eqref{eq:l1}.
\end{proof}

\section{Simultaneously coupling multiple monomial Stratonovich integrals}\label{sec:coup-multiple}
This section describes the construction of a successful coupling based on
a driving \(2\)-dimensional Brownian motion $(W_1, W_2)$ and the complete
finite set of 
monomial stochastic integrals up to a given scaling degree $n$, given by $(I_{(a,b)}: a \ge 1,\; b \ge 0,\; a + b \le n)$.
The construction uses an inductive strategy; coupling first
at the level of monomial stochastic integrals $I_{(k,l)}$ for all $(k,l) \prec (a,b)$ 
and then coupling $I_{(a,b)}$ while ensuring that the lower order integrals do not deviate too far from coupled agreement.

Recall
$$
\mathbf{X}_{(a,b)}
\quad=\quad
\left(W_1, I_{(c,d)} ; {(c,d) \preceq (a,b)}, \ c \ge 1\right)\,.
$$
We will abbreviate the complete set of monomial stochastic integrals (up to \(I_{(0,n)}\)) as
$$
\X(t) \quad=\quad \X_{(0,n)}(t),\, \ t \ge 0\,.
$$
$\tX_{(a,b)}$ and $\widetilde{\X}$ are defined in a similar manner.

The main theorem of this article states the existence of this successful coupling and estimates the rate at which it happens. In the following, we will need a simple norm on quantities such as $\X$; we use Euclidean norm viewing $\X$ as a vector in the Euclidean space of appropriate dimension.
\begin{thm}\label{thm:couplemain}
For any pair of starting points $\X(0)$ and $\tX(0)$
there exists a successful Markovian coupling construction $\mathbb{P}$ of $\X$ and $\tX$, 
with coupling time $\CouplingTime$ satisfying the following rate estimate:\\
There are positive constants $C$, $\gamma$ such that if \(t\geq1\) then
$$
\sup\{\Prob{\CouplingTime > t}\;:\; |\X(0)| \le 1\,,\; |\tX(0)| \le 1\} \quad\le\quad Ct^{-\gamma}\,.
$$
\end{thm}
\begin{proof}
As before, $C, \gamma$ will denote generic positive constants whose values will change from line to line.
The constant $R>1$ is a tuning parameter for the coupling construction:
its value will be specified later.

By a combination of reflection coupling and then synchronous coupling,
we may assume that the starting points satisfy $(W_1, W_2)(0)=(\tW_1, \tW_2)(0)$ and $W_2(0)=RW_1(0)$.
We will write this as $(\X(0), \tX(0)) \in \mathcal{R}$ where
\begin{multline*}
\mathcal{R}\;=\;
\left\lbrace({w},\widetilde{{w}}) : (w_1,w_2) = (\widetilde{w}_1,\widetilde{w}_2), w_2=Rw_1, 
\text{ remaining coordinates of }\right.\\
\left. {w}, \widetilde{{w}} \text{ unconstrained} \right\rbrace \,.
\end{multline*}
%
At the end of the proof we will check that the rate of coupling is not affected by the time taken to arrange for this.

\medskip
The main body of the proof is based on induction on the number of \(\preceq\)-ordered monomial Stratonovich
integrals to be coupled.
\\
\noindent
\emph{Induction hypothesis: } 
Define $\Delta  \X_{(a,b)} = \X_{(a,b)}-\tX_{(a,b)}$. For any $(a,b) \in \Delta_n$, there exists a successful Markovian coupling between 
the arrays of monomial Stratonovich integrals
$\X_{(a,b)}$ and $\tX_{(a,b)}$, and between \(W\) and \(\tW\),
with coupling time $T_{(a,b)}$ such that for all \(t\ge1\)
$$
\sup
\{\Prob{T_{(a,b)} >t} \;:\; |\Delta \X_{(a,b)}(0)| \le 1,\; (\X(0), \tX(0)) \in \mathcal{R}\}
\quad\le\quad Ct^{-\gamma}\,.
$$
for positive constants $C$, $\gamma$. 

By Lemma \ref{lem:scaling} there is no loss of generality 
in assuming that (a) the starting points $\X_{(a,b)}(0)$ and $\tX_{(a,b)}(0)$ 
satisfy $(\X(0), \tX(0)) \in \mathcal{R}$ and (b) $|\Delta \X_{(a,b)}(0)| =1$.

Lemma \ref{lem:Heisenberg} establishes the inductive hypothesis in the initial case of $(a,b)=(1,0)$,
since then 
$(W_1, W_2)(0)=(\tW_1, \tW_2)(0)$ and $|\Delta I_{(1,0)}(0)| \le |\Delta \X_{(a,b)}(0)| =1$.

Consider $(a,b) \in \Delta_n$ such that $(1,0) \prec (a,b)$ (equivalently, $a+b>1$).
Let $(a^-,b^-)$ be the $\preceq$ predecessor of $(a,b)$.
The inductive step of the proof is as follows:
suppose the induction hypothesis is true for $(a^-,b^-)$;
then it is required to show that the hypothesis is also true for $(a,b)$.
The key to this is to conduct a careful analysis of the cycles described informally above.
By scaling arguments, it is sufficient to do this for the first cycle,
and then to show how scaling arguments can be used to establish suitable convergence over the whole sequence of cycles.

If $a=0$, then from the definition of $\X_{(a,b)}$, $\X_{(a,b)}=\X_{(a^-,b^-)}$ (as remarked in Section \ref{sec:TP}) and there is nothing to prove. Therefore, we assume $a \ge 1$.
 
\medskip\noindent
\textbf{A: Description of the first cycle}\\
We can write $\X_{(a,b)}=( \X_{(a^-,b^-)}, I_{(a,b)})$.
The three phases of the first cycle have end-points given by the following stopping times.
\begin{tabbing}
1: \= \(\sigma_1 = \inf\{ t \ge 0: \Delta \X_{(a^-,b^-)}(t)=0\}\), 
\quad \= Coupling of \(\X_{(a^-,b^-)}\) and \(\tX_{(a^-,b^-)}\)
\\
\>\> derived from inductive hypothesis,
\\
\>\> and note \(\X_{(a^-,b^-)}(\sigma_1)=\tX_{(a^-,b^-)}(\sigma_1)\);
\\
2: \> \(\sigma_2 = \inf\{ t \ge \sigma_1: \ R W_1(t) = W_2(t)\}\) \> synchronous coupling till \(R W_1 = W_2\);
\\
3: \> \(\sigma_3 > \sigma_2\)
\> Coupling strategy of Lemma \ref{lem:couple1}
\\
\>\>  after re-scaling
\(\X_{(a,b)}\), \(\tX_{(a,b)}\)
\\ 
\>\> using \(|\Delta I_{(a,b)}(\sigma_1)|^{-1/(a+b+1)}\)
\\
\>\> and note \((W_1, W_2)(\sigma_3)=(\tW_1,\tW_2)(\sigma_3)\),\\
\>\> \(I_{(a,b)}(\sigma_3) = \tI_{(a,b)}(\sigma_3)\).
\end{tabbing}

Between $\sigma_1$ and $\sigma_2$, the two Brownian motions are coalesced
and synchronously coupled.
Therefore $\Delta \X_{(a,b)}(\sigma_2)=(0, \Delta I_{(a,b)}(\sigma_1))$.

\medskip\noindent
\textbf{Phase 1: }
At the end of this phase $\Delta \X_{(a,b)}(\sigma_1)=(0, \Delta I_{(a,b)}(\sigma_1))$.
By the induction hypothesis
\begin{align}\label{eq:hold}
\Prob{\sigma_1 >t} \quad\le\quad Ct^{-\gamma} \qquad \text{ for }t\geq1\,.
\end{align}
We need a tail bound on $\Prob{|\Delta I_{(a,b)}(\sigma_1)| > x}$ for $x \ge 2$.
Using \eqref{eq:hold}, $x \ge 2$ and $t \ge 1$,
\begin{align}\label{eq:hula1}
\Prob{|\Delta I_{(a,b)}(\sigma_1)| > x} 
\quad& \le\quad \Prob{\sigma_1>t} + \Prob{|\Delta I_{(a,b)}(\sigma_1)| > x, \; \sigma_1 \le t}\nonumber\\
\quad& \le\quad Ct^{-\gamma} + \Prob{|\Delta I_{(a,b)}(\sigma_1)| > x, \; \sigma_1 \le t}\,.
\end{align}
Since \(x \ge 2\) and \(|\Delta I_{(a,b)}(0)| \le 1\), the second probability satisfies
\begin{multline}\label{eq:hulahu}
 \Prob{|\Delta I_{(a,b)}(\sigma_1)| > x, \;\sigma_1 \le t} \quad\le\quad
 \\
 \Prob{\sup_{u \le t}\left|\int_0^u W^a_1(s)W^b_2(s)\circ {\d}W_2(s) - \int_0^u \widetilde{W}^a_1(s)W^b_2(s)\circ {\d}W_2(s) + \Delta I_{(a,b)}(0)\right| >x}
 \\
 \quad\le\quad 2\Prob{\sup_{u \le t}\left|\int_0^u W^a_1(s)W^b_2(s) \circ {\d}W_2(s)\right| >x/4}\,.
\end{multline}
Using the It\^o representation of $I_{(a,b)}$ (Equation \eqref{eq:itostrat}),
\begin{multline*}
\sup_{u \le t}\left|\int_0^u W^a_1(s)W^b_2(s) \circ {\d}W_2(s)\right| 
\quad\le\quad
\\
\sup_{u \le t}\left|\int_0^u W^a_1(s)W^b_2(s){\d}W_2(s)\right|
+ \frac{b}{2}\sup_{u \le t}\left|\int_0^u W^a_1(s)W^{b-1}_2(s) {\d}s\right|\,. 
\end{multline*}
Now Cauchy-Schwarz arguments yield
\begin{align*}
&\Expect{\sup_{u \le t}\left|\int_0^u W^a_1(s)W^b_2(s) \circ {\d}W_2(s)\right|^2}\\
& \quad \le \quad 
2 \Expect{\sup_{u \le t}\left|\int_0^u W^a_1(s)W^b_2(s) {\d}W_2(s)\right|^2}
 + \frac{b^2}{2}\Expect{\sup_{u \le t}\left|\int_0^u W^a_1(s)W^{b-1}_2(s) {\d}s\right|^2}\\
& \quad \le \quad
2 \Expect{\sup_{u \le t}\left|\int_0^u W^a_1(s)W^b_2(s) {\d}W_2(s)\right|^2}
 + \frac{b^2}{2}\Expect{\left(\int_0^t |W^a_1(s)||W^{b-1}_2(s)| {\d}s\right)^2}\,.
\end{align*}
By the BDG inequality,
\begin{multline*}
\Expect{\sup_{u \le t}\left|\int_0^u W^a_1(s)W^b_2(s) {\d}W_2(s)\right|^2}
\\
\quad\le\quad C \Expect{\int_0^t W^{2a}_1(s)W^{2b}_2(s) {\d}s}
\quad\le\quad C\int_0^t s^{a}s^{b} {\d}s 
\quad=\quad \frac{C t^{a+b+1}}{a+b+1}\,,
\end{multline*}
while a further application of the Cauchy Schwarz inequality yields
\begin{multline*}
\Expect{\left(\int_0^t |W^a_1(s)||W^{b-1}_2(s)| {\d}s\right)^2} 
\quad\le\quad t \Expect{\int_0^t W^{2a}_1(s)W^{2b-2}_2(s) {\d}s}\\
\quad\le\quad Ct\int_0^ts^as^{b-1}{\d}s 
\quad=\quad \frac{C t^{a+b+1}}{a+b}\,.
\end{multline*}
Using the Tchebychev inequality and the above two bounds together with Equation \eqref{eq:hulahu}, if $x \ge 2$ and $t \ge 1$ then
\begin{align}\label{eq:hula2}
\Prob{|\Delta I_{(a,b)}(\sigma_1)| > x, \; \sigma_1 \le t} 
\quad\le\quad \frac{Ct^{a+b+1}}{x^2}\,.
\end{align}
Combining inequalities \eqref{eq:hula1} and \eqref{eq:hula2},
and choosing
$t=x^{2/(a+b+1+\gamma)}$,
if \(x\geq2\) then
\begin{align}\label{eq:ibound}
\Prob{|\Delta I_{(a,b)}(\sigma_1)| > x} 
\quad\le\quad Cx^{-2\gamma / (a+b + 1 + \gamma)}\,.
\end{align}

\medskip\noindent
\textbf{Phase 2: }
During this phase, the driving Brownian motions are synchronously coupled till $(W_1,W_2)$ hits the line $y=Rx$.
This is done to get to the starting configuration of the coupled processes in Lemma \ref{lem:couple1}.
Between $\sigma_1$ and $\sigma_2$, 
the two Brownian motions are coalesced and synchronously coupled and 
hence $\Delta \X_{(a,b)}(\sigma_2)=(0, \Delta I_{(a,b)}(\sigma_1))$.

To get a bound on the tail of the distribution of $\sigma_2-\sigma_1$, we rewrite it as follows, using \eqref{eq:hold}:
\begin{align*}
\Prob{\sigma_2-\sigma_1 >t} \quad&\le\quad
\Prob{\sigma_1 > t^{\alpha}} + \Prob{\sigma_2-\sigma_1 >t, \sigma_1 \le t^{\alpha}}\\
\quad&\le\quad Ct^{-\alpha\gamma} + \Prob{\sigma_2-\sigma_1 >t, \sigma_1 \le t^{\alpha}}\,,
\end{align*}
for $t \ge 1$,
and arbitrary $\alpha \in (0,1)$.
The second term above can be estimated
in terms of the distance of $(W_1,W_2)$ from the line $y=Rx$ at time $\sigma_1$,
in fact following the lines of the proof of \eqref{eq:betatail}: 
\begin{align}\label{eq:boundrep}
\Prob{\sigma_2-\sigma_1 >t, \;\sigma_1 \le t^{\alpha}} 
\quad&\le\quad \Prob{|(W_1,W_2)(\sigma_1)- (W_1,W_2)(0)| >x, \; \sigma_1 \le t^{\alpha}}\nonumber\\
& \qquad + \Prob{|(W_1,W_2)(\sigma_1) - (W_1,W_2)(0)| \le x, \; \sigma_2 - \sigma_1 >t}\,,
\end{align} 
where $x, \alpha >0$ will be chosen appropriately to optimize the bounds.
To estimate the first probability in \eqref{eq:boundrep}, note that
\begin{multline*}
\Prob{|(W_1,W_2)(\sigma_1)- (W_1,W_2)(0)| >x, \;\sigma_1 \le t^{\alpha}}\\
\quad\le\quad \Prob{\sup_{s \in[0, t^{\alpha}]}|(W_1,W_2)(s) - (W_1,W_2)(0)| >x}
\quad \le\quad Ct^{\alpha/2}x^{-1}\,.
\end{multline*}
To control the second probability in \eqref{eq:boundrep}, 
condition on the event $[|(W_1,W_2)(\sigma_1) - (W_1,W_2)(0)| \le x]$ 
and use the strong Markov property to argue that 
the hitting time on the line $y=Rx$ by the Brownian motion $((W_1,W_2)(t) - (W_1,W_2)(\sigma_1): t \ge \sigma_1)$ 
is stochastically dominated by the hitting time on zero by a one dimensional Brownian motion starting from $x$.
Therefore,
$$
\Prob{|(W_1,W_2)(\sigma_1) - (W_1,W_2)(0)| \le x, \sigma_2 - \sigma_1 >t} \quad\le\quad C\frac{x}{\sqrt{t}}\,.
$$
Using the above estimates in \eqref{eq:boundrep}, we obtain
$$
\Prob{\sigma_2-\sigma_1 >t, \; \sigma_1 \le t^{\alpha}} \quad\le\quad C\frac{t^{\alpha/2}}{x} + C\frac{x}{\sqrt{t}}\,.
$$
Using this and \eqref{eq:hold}, and choosing suitable values of $x$ and $\alpha$, we obtain $\gamma>0$ such that
\begin{align}\label{eq:hold2}
\Prob{\sigma_2 - \sigma_1>t} \quad\le\quad Ct^{-\gamma} \quad \text{ for } t \ge 1\,.
\end{align}

\medskip\noindent
\textbf{Phase 3: }
In this phase, 
Lemma \ref{lem:couple1} is used to couple 
$(W_1,W_2,I_{(a,b)})$ with $(\tW_1,\tW_2,\tI_{(a,b)})$ 
while controlling the difference between the lower order integrals of the coupled processes.

Note that at time $\sigma_3$ the array \(\Delta \X_{(a,b)}(\sigma_3)\) 
of monomial Stratonovich integrals is obtained by appending
\(\Delta I_{(a,b)}(\sigma_3)=0\)
to the array \(\Delta \X_{(a^-,b^-)}(\sigma_3)\).
Now $\Delta \X_{(a^-,b^-)}(\sigma_3)$ 
will be the discrepancy between the coupled sets of integrals at the end of the third phase and thus,
it is necessary to control its size.
We do this by controlling the size (in an appropriate sense) of each individual integral appearing in $\Delta \X_{(a^-,b^-)}(\sigma_3)$
and showing 
the coupling strategy of Lemma \ref{lem:couple1} does not make this size large.
Fix any $(k,l) \preceq (a^-,b^-)$.
Note that, as $\Delta I_{(k,l)}(\sigma_2)=\Delta I_{(k,l)}(\sigma_1)=0$, 
scaling yields the following distributional equality (where the second equality simply
involves rewriting the Stratonovich integral as the sum of an It\^o integral and a time integral):
\begin{multline}\label{eq:disteq}
\Delta I_{(k,l)}(\sigma_3)\quad\StochEquiv\quad 
U\int_0^{T_{R,(a,b)}} (B^k_1(s) -\tB^k_1(s))B^l_2(s) \circ {\d}B_2(s)
\\
\quad=\quad 
U\left(\int_0^{T_{R,(a,b)}} (B^k_1(s) -\tB^k_1(s))B^l_2(s) {\d}B_2(s)
 + \frac{l}{2}\int_0^{T_{R,(a,b)}} (B^k_1(s) -\tB^k_1(s))B^{l-1}_2(s){\d}s\right)
\end{multline}
where $U$ has the same distribution as $|\Delta I_{(a,b)} (\sigma_1)|^{(k+l+1)/(a+b+1)}$, 
and $(B_1, B_2)$ and $(\tB_1, \tB_2)$ are two-dimensional Brownian motions 
starting respectively from $$\left(\frac{W_1(\sigma_2)}{|\Delta I_{(a,b)} (\sigma_1)|^{1/(a+b+1)}}, 
\frac{W_2(\sigma_2)}{|\Delta I_{(a,b)} (\sigma_1)|^{1/(a+b+1)}}\right)$$
and 
$$\left(\frac{\tW_1(\sigma_2)}{|\Delta I_{(a,b)} (\sigma_1)|^{1/(a+b+1)}}, 
\frac{\tW_2(\sigma_2)}{|\Delta I_{(a,b)} (\sigma_1)|^{1/(a+b+1)}}\right),$$ 
and $T_{R,(a,b)}$ is the coupling time for the coupling construction of $$\left(B_1, B_2, \frac{I_{(a,b)}(\sigma_2)}{|\Delta I_{(a,b)} (\sigma_1)|} + \int B^a_1 B^b_2 \circ {\d}B_2\right)$$ and $$\left(\tB_1, \tB_2, \frac{\tI_{(a,b)}(\sigma_2)}{|\Delta I_{(a,b)} (\sigma_1)|} + \int \tB^a_1 \tB^b_2 \circ {\d}\tB_2\right)$$ 
given in Lemma \ref{lem:couple1}.
Furthermore, $U$ is independent of 
$((B_1(t)-B_1(0), B_2(t)-B_2(0)): t \ge 0)$ and $((\tB_1(t)-\tB_1(0), \tB_2(t)-\tB_2(0)): t \ge 0)$.

Define stopping times $\theta_j, \tau_j, \eta_j, \lambda_j, \beta_j, j \ge 1$ in the time interval $[0, T_{R,(a,b)}]$ 
as in the proof of Lemma \ref{lem:couple1}.
As the Brownian motions move together on the intervals $[\beta_{j-1}, \theta_j]$ and $[\lambda_j, \beta_j]$, 
the monomial Stratonovich integral $\Delta I_{(k,l)}$ does not change on these intervals.

On $[\theta_1, \lambda_1]$, the It\^o integral in \eqref{eq:disteq} can be written as
\begin{multline}\label{eq:dec1}
\int_{\theta_1}^{\lambda_1} (B^k_1(s) -\tB^k_1(s))B^l_2(s) {\d}B_2(s) 
\quad=\quad \int_{\theta_1}^{\lambda_1} \Delta B_1(s)\sum_{j=1}^{k-1}B^{k-j}_1(s)\tB^{j-1}_1(s)B^l_2(s) {\d}B_2(s)
\\
\quad=\quad
\int_{\theta_1}^{\lambda_1} (\Delta B_1(s)B_1^{k+l-1}(\theta_1)R^l)
  \sum_{j=1}^{k-1}\left(\frac{B_1(s)}{B_1(\theta_1)}\right)^{k-j}
    \left(\frac{\tB_1(s)}{\tB_1(\theta_1)}\right)^{j-1}
    \left(\frac{B_2(s)}{B_2(\theta_1)}\right)^l 
{\d}B_2(s)\\
\quad=\quad
\left( \int_{\theta_1}^{\lambda_1} (\Delta B_1(s)B_1^{a+b-1}(\theta_1)R^b)
  \sum_{j=1}^{k-1}\left(\frac{B_1(s)}{B_1(\theta_1)}\right)^{k-j}
    \left(\frac{\tB_1(s)}{\tB_1(\theta_1)}\right)^{j-1}
    \left(\frac{B_2(s)}{B_2(\theta_1)}\right)^l 
{\d}B_2(s)\right) \\
 \quad \quad \quad \quad \times \frac{1}{(B_1(\theta_1))^{(a+b)-(k+l)}R^{b-l}}\,.
\end{multline}
Using $\sup_{t \in [\theta_1, \tau_1]}|\Delta B_1(t)| =  \frac{1}{|B_1(\theta_1)|^{a+b-1}R^b}$ and
$\sup_{t \in [\tau_1, \eta_1]}|\Delta B_1(t)|= |\Delta B_1(\tau_1)|$
together with \eqref{eq:l1},
\begin{align}\label{align:ie1}
\Prob{\sup_{\theta_1 \le t \le \lambda_1} |\Delta B_1(s)||B_1^{a+b-1}(\theta_1)|R^b \ge x} \quad\le\quad \frac{1}{x}
\qquad\text{ for }x \ge 1\,.
\end{align}
The same line of argument employed to obtain
\eqref{eq:tauflucold} 
can be used to get the following bound. For $1 \le j \le k-1$, $x > 2^{k+l-1}$:
\begin{align*}
\Prob{\sup_{\theta_1 \le s \le \tau_1}
\left|\frac{W_1(s)}{W_1(\theta_1)}\right|^{k-j}
    \left|\frac{\tW_2(s)}{\tW_1(\theta_1)}\right|^{j-1}
    \left|\frac{W_2(s)}{ W_2(\theta_1)}\right|^l > x} 
\quad\le\quad 
\frac{2^{2(k+l-1)/l}}{R^{8n+4}x^{2/l}} \; C\,.
\end{align*}
(where the above bound is taken to be zero if $l=0$).
As a result, using \eqref{eq:bound2} and \eqref{eq:l2} (bounding
$\sup_{\tau_1 \le t \le \eta_1}\left|\frac{B_1(t)}{B_1(\theta_1)}\right|$ and 
$\sup_{\eta_1 \le t \le \lambda_1}\left|\frac{B_1(t)}{B_1(\theta_1)}\right|$;
the same bounds can be shown to hold when $B_1$ is replaced by $\tB_1$ and $B_2$), 
we can pick $\gamma>0$ such that, for any $1 \le j \le k-1$ and all $x > 2^{k+l-1}$,
\begin{align}\label{align:ie2}
\Prob{\sup_{\theta_1 \le t \le \lambda_1}
  \left|\frac{B_1(s)}{B_1(\theta_1)}\right|^{k-j}
  \left|\frac{\tB_1(s)}{\tB_1(\theta_1)}\right|^{j-1}
  \left|\frac{B_2(s)}{B_2(\theta_1)}\right|^l > x} 
\quad\le\quad
\frac{C}{(R^{2n}x)^{\gamma}}\,.
\end{align}
\eqref{align:ie1} and \eqref{align:ie2} together imply that there exists $\gamma>0$ such that for any $1 \le j \le k-1$ and all $x > 2^{k+l}$,
\begin{align}\label{eq:haha1}
\Prob{\sup_{\theta_1 \le s \le \lambda_1}|\Delta B_1(s)|
|B^{a+b-1}(\theta_1)|
R^b\left|\frac{B_1(s)}{B_1(\theta_1)}\right|^{k-j}
\left|\frac{\tB_1(s)}{\tB_1(\theta_1)}\right|^{j-1}
\left|\frac{B_2(s)}{B_2(\theta_1)}\right|^l > x} 
\quad\le\quad \frac{C}{x^{\gamma}}\,.
\end{align}
Also, recall from \eqref{eq:tailhalf} (using $\alpha=1$) that
\begin{align}\label{eq:haha2}
\Prob{\lambda_1-\theta_1 >t} \quad\le\quad Ct^{-1/2} \qquad \text{ for }t \ge 1\,.
\end{align}
Using \eqref{eq:haha1} and \eqref{eq:haha2} in Lemma \ref{lem:tail1}, we can chose positive $\gamma, x_0$ such that
(for $x \ge x_0$)
\begin{multline}\label{eq:dec2}
\Prob{\left|\int_{\theta_1}^{\lambda_1} (\Delta B_1(s)B^{a+b-1}(\theta_1)R^b)\sum_{j=1}^{k-1}\left(\frac{B_1(s)}{B_1(\theta_1)}\right)^{k-j}\left(\frac{\tB_1(s)}{\tB_1(\theta_1)}\right)^{j-1}\left(\frac{B_2(s)}{B_2(\theta_1)}\right)^l {\d}B_2(s)\right| >x} 
\\
\le \sum_{j=1}^{k-1}\Prob{\left|\int_{\theta_1}^{\lambda_1} (\Delta B_1(s)B^{a+b-1}(\theta_1)R^b)\left(\frac{B_1(s)}{B_1(\theta_1)}\right)^{k-j}\left(\frac{\tB_1(s)}{\tB_1(\theta_1)}\right)^{j-1}\left(\frac{B_2(s)}{B_2(\theta_1)}\right)^l {\d}B_2(s)\right| >\frac{x}{k}}
\\ 
\quad\le\quad Cx^{-\gamma}\,.
\end{multline}
Furthermore, by virtue of the definition of the stopping time $\theta_1$ in the coupling construction,
and because $(k,l) \prec (a,b)$,
\begin{align}\label{eq:dec3}
\frac{1}{|(B_1(\theta_1))|^{(a+b)-(k+l)}R^{b-l}} 
\quad\le\quad \frac{1}{R^{f(a,b)-f(k,l)}} 
\quad\le\quad \frac{1}{R}
\,.
\end{align}
Using \eqref{eq:dec2} and \eqref{eq:dec3} in \eqref{eq:dec1}, for $x \ge x_0/R$ we obtain
\begin{align}\label{eq:lala1}
\Prob{\left|\int_{\theta_1}^{\lambda_1} (B^k_1(s) -\tB^k_1(s))B^l_2(s) {\d}B_2(s)\right| >x} 
\quad\le\quad \frac{C}{(Rx)^{\gamma}}\,.
\end{align}
Recall from Lemma \ref{lem:couple1} that $S_{R,(a,b)} =\bigcup_{j=1}^{\infty}[\theta_j,\lambda_j]$ and
$$
\int_{0}^{T_{R,(a,b)}} (B^k_1(s) -\tB^k_1(s))B^l_2(s) {\d}B_2(s)= \int_{S_{R,(a,b)}}(B^k_1(s) -\tB^k_1(s))B^l_2(s) {\d}B_2(s).
$$
Now apply Lemma \ref{lem:tail2} to the sum on the right hand side of
$$
R\int_{0}^{T_{R,(a,b)}} (B^k_1(s) -\tB^k_1(s))B^l_2(s) {\d}B_2(s)= \tau^{**}_1 + \sum_{j=1}^{\infty} R^{-(k+l+1)j\delta/(a+b+1)}\left(\Pi_{m=1}^jX^{**}_m\right)\tau^{**}_{j+1}
$$
with $(X^{**}_j,\tau^{**}_j)$ replacing $(X_j,\tau_j)$ in the statement of the lemma, where
$$
\tau^{**}_j\quad=\quad
\frac{R\left|\int_{\theta_j}^{\lambda_j} (B^k_1(s) -\tB^k_1(s))B^l_2(s) {\d}B_2(s)\right|}
  {|\Delta I_{(a,b)}(\beta_{j-1})|^{(k+l+1)/(a+b+1)}}
$$
and
$$
X^{**}_j\quad=\quad
\frac{R^{(k+l+1)\delta/(a+b+1)}
  |\Delta I_{(a,b)}(\beta_{j})|^{(k+l+1)/(a+b+1)}}
  {|\Delta I_{(a,b)}(\beta_{j-1})|^{(k+l+1)/(a+b+1)}}\,.
$$
Taking $\delta$ to be the same as that used in \eqref{eq:coupleclose1}, 
and setting $\beta_0=0$, we can find $\gamma' >0$ for which
\begin{align}\label{eq:ibound2}
\Prob{\left|\int_{0}^{T_{R,(a,b)}} (B^k_1(s) -\tB^k_1(s))B^l_2(s) {\d}B_2(s)\right| > x} 
\quad\le\quad \frac{C}{(Rx)^{\gamma'}}
\end{align}
for $x \ge x_0/R$, when $R$ is sufficiently large, where $x_0$ is the same as that used in \eqref{eq:lala1}.
Using the same estimates as above and using the assertions of Lemma \ref{lem:tail1} involving time integrals, 
for $l \ge 1$ it follows in a similar way that there is $\gamma''>0, x_1 >0$ such that
\begin{align}\label{eq:bound2p}
\Prob{\left|\int_0^{T_{R,(a,b)}} (B^k_1(s) -\tB^k_1(s))B^{l-1}_2(s){\d}s\right|>x} 
\quad\le\quad \frac{C}{(R^{2n+2}x)^{\gamma''}}
\end{align}
for $x \ge x_1/R^{2n+2}$, for $R$ sufficiently large.

Using the distributional equality \eqref{eq:disteq}, when $x \ge \max\{x_0^2,x_1^2,4\}/R$ 
and \(U\) is as defined for that equation,
we can find positive constants $\gamma, \gamma_1, \gamma_2$ such that
\begin{multline*}
\Prob{|\Delta I_{(k,l)}(\sigma_3)| > x} \quad=\quad 
\Prob{U
 \left|\int_0^{T_{R,(a,b)}} (B^k_1(s) -\tB^k_1(s))B^l_2(s) \circ {\d}B_2(s)\right| >x}\\
\quad\le\quad
\Prob{U \ge \sqrt{xR}} + 
\Prob{\left|\int_{0}^{T_{R,(a,b)}} (B^k_1(s) -\tB^k_1(s))B^l_2(s) \circ {\d}B_2(s)\right| > \sqrt{x/R}}\\
\quad=\quad 
\Prob{|\Delta I_{(a,b)}(\sigma_1)| > (\sqrt{xR})^{(a+b+1)/(k+l+1)}}\\
 \quad \quad \quad
 + \Prob{\left|\int_{0}^{T_{R,(a,b)}} (B^k_1(s) -\tB^k_1(s))B^l_2(s) \circ {\d}B_2(s)\right| > \sqrt{x/R}}\\
\quad \le\quad \frac{C}{(xR)^{\gamma_1} }+ \frac{C}{(xR)^{\gamma_2}} \quad\le\quad \frac{C}{(xR)^{\gamma}}\,.
\end{multline*}
The last step above follows from \eqref{eq:ibound}, \eqref{eq:ibound2} and \eqref{eq:bound2p}.
As this bound can be chosen to hold for all $(k,l) \prec (a,b)$ and $\Delta I_{(a,b)}(\sigma_3)=0$, 
we can choose positive constants $C, \gamma$ such that
\begin{align}\label{eq:mb1}
\Prob{|\Delta \X_{(a,b)}(\sigma_3)| > x} \quad\le\quad \frac{C}{(xR)^{\gamma}} \qquad \text{ when } x \ge C/R\,.
\end{align}
A bound on the tail of the law of $\sigma_3$ can be obtained using Lemma \ref{lem:couple1}:
if \(t\ge 1\) then
$$
\Prob{(\sigma_3-\sigma_2)/|\Delta I_{(a,b)}(\sigma_1)|^{2/(a+b+1)}>R^{4n+2}t} 
\quad\le\quad Ct^{-\gamma}\,.
$$
Together with \eqref{eq:ibound} this implies that when \(t\ge1\)
\begin{align}\label{eq:hold3}
\Prob{\sigma_3 - \sigma_2>R^{4n+2}t} \quad\le\quad Ct^{-\gamma}\,.
\end{align}
Thus, from \eqref{eq:hold}, \eqref{eq:hold2} and \eqref{eq:hold3}, when \(t\ge1\)
\begin{align}\label{eq:mb2}
\Prob{\sigma_3>R^{4n+2}t} \quad\le\quad Ct^{-\gamma}\,.
\end{align}

\medskip\noindent
\textbf{B: Describing subsequent cycles and successful coupling}

\noindent
After completion of the first cycle, at time $\sigma_3$, 
we re-scale $\X_{(a,b)}$ and $\tX_{(a,b)}$ according to Lemma \ref{lem:scaling} by a (random) scaling $\mathcal{S}_{R_1}$ 
such that 
$$
|\mathcal{S}_{R_1} \left( \Delta \X_{(a,b)}\right)(\sigma_3)| \quad=\quad 1\,.
$$ 
Define $\sigma_4, \sigma_5, \sigma_6$ (for the original process)
corresponding to $\sigma_1, \sigma_2, \sigma_3$ for the coupled process after scaling exactly as before, and so on.
At each stopping time $\sigma_{3k}$, $k \ge 1$, 
we denote by $\mathcal{S}_{R_k}$ 
the (random) scaling that renormalizes at \(1\) the norm of the difference of the re-scaled processes.

For any $r \ge 1$, $t \ge 0$,
$$
r|\Delta \X_{(a,b)}(t)| 
\quad\le\quad |\mathcal{S}_{r} \left(\Delta \X_{(a,b)}\right)(t)| 
\quad\le\quad r^{a+b+1}|\Delta \X_{(a,b)}(t)|\,,
$$
with inequalities reversed if $r \le 1$.

For each $k \ge 1$,
$$
\left|\mathcal{S}_{(\Pi_{j=1}^k R_j)} \left(\Delta \X_{(a,b)}\right)(\sigma_{3k})\right| \quad=\quad 1\,.
$$
Therefore, 
if it can be shown that $\lim_{k \rightarrow \infty}\Pi_{j=1}^k R_j = \infty$
then it follows that $$\lim_{k \rightarrow \infty}|\Delta X_{(a,b)}(\sigma_{3k})|=0.$$ 
We achieve this by estimating the tail of the distribution of $R_1^{-1}$.
Note that if $R_1^{-1} \le1$, then from the above relations
$$
\frac{R^{1/(a+b+1)}}{R_1}\quad\le\quad |R\Delta \X_{(a,b)}(\sigma_{3})|^{1/(a+b+1)}
$$
while if $R_1^{-1} \ge 1$ then
$$
\frac{R}{R_1} \quad\le\quad |R\Delta \X_{(a,b)}(\sigma_{3})|\,.
$$
Thus, for $x \ge 1$, if $R^{1/(a+b+1)} > x$ then
\begin{align*}
\Prob{\frac{R^{1/(a+b+1)}}{R_1} \ge x} \quad& =\quad \Prob{\frac{R^{1/(a+b+1)}}{R_1} \in  [x, R^{1/(a+b+1)}]} + \Prob{\frac{R^{1/(a+b+1)}}{R_1} > R^{1/(a+b+1)}}\\
\quad&\le\quad \Prob{\frac{R^{1/(a+b+1)}}{R_1} \in  [x, R^{1/(a+b+1)}]} + \Prob{\frac{R^{1/(a+b+1)}}{R_1} > x, \; R_1^{-1}>1}\\
\quad&\le\quad \Prob{\frac{R^{1/(a+b+1)}}{R_1} \in [x, R^{1/(a+b+1)}]} + \Prob{\frac{R}{R_1} > x, \; R_1^{-1}>1}\\
\quad&\le\quad \Prob{|R\Delta \X_{(a,b)}(\sigma_3)| \ge x^{a+b+1}} + \Prob{|R\Delta \X_{(a,b)}(\sigma_3)| \ge x} 
\\
\quad\le\quad Cx^{-\gamma}\,,
\end{align*}
where the last inequality is a consequence of \eqref{eq:mb1}.

On the other hand, if $R^{1/(a+b+1)} \le x$,
\begin{align*}
\Prob{\frac{R^{1/(a+b+1)}}{R_1} \ge x} \quad& =\quad \Prob{\frac{R^{1/(a+b+1)}}{R_1} \ge x, \; R_1^{-1} \ge 1}\\
\quad& \le\quad \Prob{\frac{R}{R_1} \ge x, \; R_1^{-1} \ge 1}
 \quad\le\quad \Prob{|R\Delta \X_{(a,b)}(\sigma_3)| \ge x}\\ 
\quad&\le\quad Cx^{-\gamma}\,.
\end{align*}
Combining the above two bounds, if \(x \ge 1\) then
\begin{align}\label{eq:holdX}
\Prob{\frac{R^{1/(a+b+1)}}{R_1} \ge x} \quad\le\quad Cx^{-\gamma}\,.
\end{align}
Applying Lemma \ref{lem:tail2} with $\displaystyle{X_k = \frac{R^{1/(a+b+1)}}{R_k}}$ and $\tau_k=1$, we obtain
$$
\sum_{k=1}^{\infty}R^{-k/(a+b+1)}(\Pi_{j=1}^k X_j) \quad<\quad \infty
$$
almost surely for sufficiently large $R$.
In particular, $\lim_{k \rightarrow \infty}\Pi_{j=1}^k R_j = \infty$ and consequently,
$$
\lim_{k \rightarrow \infty}|\Delta \X_{(a,b)}(\sigma_{3k})|=0.
$$
Finally, to show that the coupling is successful and to verify the induction hypothesis for $(a,b)$, 
it is necessary to show that $\lim_{k \rightarrow \infty} \sigma_{3k}$ is almost surely finite 
and that its law has a power law tail.

This follows by applying Lemma \ref{lem:tail2} to the sum on the right hand side of the expression
$$
\sigma_{3k+3} \quad=\quad \hat{\tau}_1 + \sum_{l=1}^k R^{-2l/(a+b+1)}(\Pi_{j=1}^l X_j^2) \hat{\tau}_{l+1}, \qquad k \ge 1,
$$
with $(X_k^2, \hat{\tau}_k)$ in place of $(X_j, \tau_j)$ in the lemma, 
where $\displaystyle{X_k = \frac{R^{1/(a+b+1)}}{R_k}}$, defined for $k \ge 1$, and 
$\displaystyle{\hat{\tau}_k=\left(\Pi_{j=1}^{k-1} R_j^2\right)\left(\sigma_{3k}-\sigma_{3k-3}\right)}$, defined for $k \ge 2$, and $\sigma_0=0$.
As the law of $\hat{\tau}_k$ has the same tail as that of $\sigma_3$, 
it follows from \eqref{eq:mb2} and \eqref{eq:holdX} that if \(t\ge1\) then
$$
\Prob{\lim_{k \rightarrow \infty} \sigma_{3k} > R^{4n+2}t} \quad\le\quad Ct^{-\gamma}\,,
$$
for sufficiently large $R$.
This establishes the induction hypothesis, and so completes the construction of a successful coupling 
when the starting points of the coupled Brownian motions satisfy $(W_1, W_2)(0)=(\tW_1, \tW_2)(0)$ and $W_2(0)=RW_1(0)$.

The argument is completed by showing how to construct the coupling from arbitrary starting points $\X(0)$ and $\tX(0)$ 
satisfying $|\X(0)| \le 1$, $|\tX(0)| \le 1$.
To do this, define the stopping times
\begin{align*}
\sigma_{-1} \quad&=\quad \inf\{t \ge 0: \text{ using reflection coupling}, \ (W_1, W_2)(t) = (\tW_1, \tW_2)(t)\}\,,\\
\sigma_0 \quad&=\quad\inf\{t \ge \sigma_{-1}: \text{ using synchronous coupling}, \ W_2(t)=RW_1(t)\}\,\\
\CouplingTime \quad&=\quad\inf\{t \ge \sigma_0: \text{ coupling strategy constructed above}, \  \X(t) = \tX(t)\}\,.
\end{align*}
Using Brownian hitting time estimates derived from the reflection principle, when \(t\ge1\)
$$
\Prob{\sigma_{-1} >t} 
\quad\le\quad \frac{|(W_1, W_2)(0) - (\tW_1, \tW_2)(0)|}{\sqrt{t}} 
\quad\le\quad \frac{2}{\sqrt{t}}\,.
$$
Now, using the fact that $|(W_1, W_2)(0) - (\tW_1, \tW_2)(0)| \le 2$,
and controlling the distance of the Brownian motion $(W_1, W_2)$ from the line $y=Rx$ at time $\sigma_{-1}$
as in the proof of \eqref{eq:betatail}, we obtain a constant $C$ that does not depend on the starting points such that
$$
\Prob{\sigma_0-\sigma_{-1} >t} \quad\le\quad C t^{-1/6}\,.
$$
Consequently, if \(t\ge1\) then
\begin{align}\label{eq:huh1}
\sup\{\Prob{\sigma_0 >t}\;:\; |\X(0)| \le 1,\; |\tX(0)| \le 1\}\quad\le\quad C t^{-1/6}\,.
\end{align}
Furthermore, for $x \ge 4$ and arbitrary $t>0$ to be chosen later,
\begin{align*}
& \Prob{|\Delta \X(\sigma_0)| >x} \quad=\quad \Prob{|\Delta \X(\sigma_{-1})| >x}\\
&\quad\le\quad\Prob{|\X(\sigma_{-1})-\X(0)| + |\tX(\sigma_{-1})-\tX(0)| + |\X(0)-\tX(0)| >x}\\
&\quad\le\quad \Prob{|\X(\sigma_{-1})-\X(0)| + |\tX(\sigma_{-1})-\tX(0)|>x/2} \quad 
(\text{ as } |\X(0)-\tX(0)| \le 2 \text{ and } x \ge 4)\\
&\quad\le\quad 2 \Prob{|\X(\sigma_{-1})-\X(0)| > x/4}\\
&\quad\le\quad 2\Prob{\sigma_{-1} >t} + 2\Prob{\sup_{s \le t} |\X(s)-\X(0)| > x/4}\\
&\quad\le\quad \frac{4}{\sqrt{t}} + C\frac{\Expect{\sum_{(k,l) \in \Delta_n} \int_0^t {W_1(s)^{2k}}{W_2(s)^{2l}}{\d}s}}{x^2}\\
& \qquad \qquad + Ct\frac{\Expect{\sum_{(k,l) \in \Delta_n} l^2\int_0^t {W_1(s)^{2k}}{W_2(s)^{2l-2}}{\d}s}}{x^2}\\
&\quad\le\quad \frac{2}{\sqrt{t}} + C\frac{\sum_{(k,l) \in \Delta_n} \int_0^t s^{k+l}{\d}s}{x^2} + Ct\frac{\sum_{(k,l) \in \Delta_n} l^2\int_0^t s^{k+l-1}{\d}s}{x^2}
\quad\le\quad \frac{2}{\sqrt{t}} + C \frac{t^{n+1}}{x^2}\\
&\quad\le\quad Cx^{-\gamma}\,.
\end{align*}
Here the last step holds
for some $\gamma>0$, and for $t$ chosen appropriately in terms of $x$.
Moreover the final constant $C$ above does not depend on the starting points 
so long as they lie in the unit ball. Hence, for $x \ge 4$,
\begin{align}\label{eq:huh2}
\sup\{\Prob{|\Delta \X(\sigma_0)| >x}\;:\; |\X(0)| \le 1, |\tX(0)| \le 1\}\quad\le\quad Cx^{-\gamma}\,.
\end{align}
Thus, for $t \ge 2$, $x \ge 4$, 
\begin{multline*}
\sup\{\Prob{\CouplingTime >t}\;:\;|\X(0)| \le 1, |\tX(0)| \le 1\} 
\quad\le\quad 
\\
\sup\{\Prob{\CouplingTime >t, \; |\Delta X(\sigma_0)| \le x, \; \sigma_0 \le t/2}\;:\;|\X(0)| \le 1, |\tX(0)| \le 1\}
\\
+ \sup\{\Prob{|\Delta X(\sigma_0)| > x}\;:\;|\X(0)| \le 1, |\tX(0)| \le 1\}
\\
+\sup\{\Prob{\sigma_0 > t/2}\;:\;|\X(0)| \le 1, |\tX(0)| \le 1\}\,.
\end{multline*}
The second and third terms above are already estimated in \eqref{eq:huh2} and \eqref{eq:huh1} respectively.
To bound the first term above, we apply the strong Markov property at $\sigma_0$ to obtain
\begin{multline*}
\sup\{\Prob{\CouplingTime >t, |\Delta X(\sigma_0)| \le x, \sigma_0 \le t/2}\;:\;|\X(0)| \le 1, |\tX(0)| \le 1\}
\quad\le\quad 
\\
\sup\{\Prob{\CouplingTime>t/2}\;:\;|\Delta \X(0)| \le x, (\X(0), \tX(0)) \in \mathcal{R}\}
\end{multline*}
Apply Lemma \ref{lem:scaling} to the right hand side above, taking $r=x^{-1}$.
Now use the fact that if $|\Delta \X(0)| \le x$, then as $x >1$,
$$
|\mathcal{S}_{x^{-1}} \left(\Delta \X\right)(0)| \quad\le\quad x^{-1} |\Delta \X(0)| \quad\le\quad 1\,.
$$
Hence, we obtain
\begin{multline*}
\sup\{\Prob{\CouplingTime>t/2}\;:\;|\Delta \X(0)| \le x, (\X(0), \tX(0)) \in \mathcal{R}\}
\quad\le\quad
\\
\sup\{\Prob{\CouplingTime>t/(2x^2)}\;:\;|\Delta \X(0)| \le 1, (\X(0), \tX(0)) \in \mathcal{R}\}\,.
\end{multline*}
If $t \ge 2x^2$ then the tail estimate of the law of the above coupling time
gives
$$
\sup\{\Prob{\CouplingTime>t/(2x^2)}\;:\;|\Delta \X(0)| \le 1, (\X(0), \tX(0)) \in \mathcal{R}\}
\quad\le\quad C\frac{x^{2\gamma}}{t^{\gamma}}\,.
$$
This estimate, along with \eqref{eq:huh1} and \eqref{eq:huh2}, yields
$$
\sup\{\Prob{\CouplingTime >t}\;:\;|\X(0)| \le 1, |\tX(0)| \le 1\}
\quad\le\quad C\frac{x^{2\gamma}}{t^{\gamma}} + Cx^{-\gamma} + C t^{-1/6}\,.
$$
Choosing $x$ appropriately in terms of $t$, we see that for sufficiently large $t$ (and consequently for all $t \ge 1$ by readjusting the constants),
$$
\sup\{\Prob{\CouplingTime >t}\;:\;|\X(0)| \le 1, |\tX(0)| \le 1\} 
\quad\le\quad Ct^{-\gamma}\,,
$$
and this proves the theorem.
\end{proof}

\section{Conclusion}
In this article, we have constructed a successful Markovian coupling for the two-dimensional Brownian motion along with a finite collection of its 
monomial Stratonovich integrals. 
In the context provided by Theorem \ref{thm:hypocouple}, 
this is a further step in the direction of extending Markovian coupling techniques 
beyond the realm of specific examples towards a more general context.
Our method shares some features with an iterative coupling scheme employed in \cite{KendallPrice-2004} for coupling iterated Kolmogorov diffusions,
though the inductive strategy described in the current paper seems to be more robust as one can build {iterations within iterations} 
into the coupling, exploiting the inductive approach described here.
A natural next step for the general program of coupling hypoelliptic diffusions
would be to couple diffusions driven by \emph{nilpotent vector fields} which do not just depend on the driving Brownian motion but the entire diffusion.
The Baker-Campbell-Hausdorff formula can be employed to show that such a coupling can be achieved 
if one can construct successful Markovian couplings for \emph{Brownian motion on the free Carnot group} of finite order \citep{Baudoin-2004}.
The geometry of the Carnot group seems to lend itself particularly to our inductive approach: 
the Lie algebra $\mathcal{U}$ of the Carnot group has a graded structure given by $\mathcal{U}=\mathcal{U}_1 \oplus \mathcal{U}_2 \dots \oplus \mathcal{U}_N$ 
and there are dilation operators $\delta_t$ that act by multiplication by $t^i$ on the elements of $\mathcal{U}_i$ while preserving the graded structure.
A possible strategy for constructing the coupling in this case would be 
to use the graded structure in the induction hypothesis and to use the dilation operator to implement the scaling strategy
used repeatedly in the above arguments.
We will investigate this in future work.

The current article also provides quantitative bounds on the distribution of the coupling time.
These can be used to obtain estimates on the total variation distance between the laws of the diffusions.
Employed in conjunction with the scaling property (Lemma \ref{lem:scaling}),
this would lead to gradient estimates for heat kernels and harmonic functions 
corresponding to the generator of the diffusion \citep{Cranston-1991,Cranston-1992,BanerjeeGordinaMariano-2016}.
We note here that 
it was shown 
in recent work \citep{BanerjeeKendall-2015,BanerjeeGordinaMariano-2016}
that optimal bounds on total variation distance and good gradient estimates, 
especially in the case of hypoelliptic diffusions, require non-immersion couplings.
However, so far, it has been possible to provide explicit constructions of these couplings only in rather special examples:
(generalized) Kolmogorov diffusions \cite{BanerjeeKendall-2015}
and Brownian motion on the Heisenberg group \cite{BanerjeeGordinaMariano-2016}.
An important challenge is to find robust non-immersion coupling constructions applicable to a wider framework and 
then to compare their performance with analogous immersion or Markovian couplings.

\textbf{Acknowledgements}
{We are grateful to the anonymous referees for their careful reading of the manuscript and suggestions that greatly improved the presentation of the article.}






\bibliographystyle{amsplain}	




\end{document}